\documentclass{article}

\usepackage{amsmath}
\usepackage{amsfonts}
\usepackage{amssymb}
\usepackage{amsthm}
\usepackage{mathrsfs}

\def\carac#1,#2{
\left[
\begin{smallmatrix}
#1 \\ #2
\end{smallmatrix}
\right]
}

\newtheorem{theorem}{Theorem}[section]

\newtheorem{lemma}[theorem]{Lemma}
\newtheorem{proposition}[theorem]{Proposition}
\newtheorem{corollary}[theorem]{Corollary}
\newtheorem{remark}[theorem]{Remark}

\newcommand{\Si}{\mathbb H} 
\newcommand{\C}{\mathbb C} 
\newcommand{\N}{\mathbb N}
\newcommand{\Z}{\mathbb Z}

\newcommand{\Q}{\mathbb Q}
\newcommand{\proj}{\mathbb P}
\newcommand{\F}{\mathbb F}
\newcommand{\End}{\mathop{\mathrm{End}}}
\newcommand{\Norm}{\mathop{\mathrm{N}}\nolimits}

\newcommand{\hsp}{\hspace{5pt}}
\newcommand{\spec}[1]{ \mathrm{Sp} (#1)}
\newcommand{\mspec}{ \mathrm{Spec}}
\newcommand{\xa}{\mathbb{A}}

\newcommand{\xz}{\mathbb{Z}}
\newcommand{\xf}{\mathbb{F}}
\newcommand{\xg}{\mathbb{G}}

\newcommand{\mat}[2]{\mbox{Mat}_{#1}(#2)}

\newcommand{\pol}{\mathscr{L}}

\newcommand{\emm}{\mathcal{M}}

\newcommand{\Sp}{\mathrm{Sp}}
\newcommand{\gtheta}{\mathscr{G}}

\title{A $p$-adic quasi-quadratic
point counting algorithm}
\author{Robert Carls, David Lubicz}
\begin{document}

\maketitle

{\small 
\begin{center}
\begin{tabular}{l@{\hspace{0.5cm}}l}
Robert Carls & David Lubicz \\
\texttt{robert.carls@uni-ulm.de} & 
\texttt{david.lubicz@univ-rennes1.fr}\\
Institute of Pure Mathematics & CELAR \\
University of Ulm & BP 7419 35174 Bruz Cedex\\
D-89069 Ulm, Germany & France \\
\end{tabular}
\end{center}
}

\bigskip
\begin{abstract}
\noindent
In this article we give an algorithm for the computation of the number
of rational points on the Jacobian variety of a generic ordinary
hyperelliptic curve defined over a finite field $\F_q$
of cardinality $q$
with time complexity $O(n^{2+o(1)})$ and space complexity
$O(n^2)$, where $n=\log(q)$.  In the latter complexity estimate the
genus and the characteristic are assumed as fixed.
Our algorithm forms a generalization of both, the AGM algorithm of J.-F.
Mestre and the canonical lifting method of T.  Satoh.  We canonically
lift a certain arithmetic invariant of the Jacobian of the
hyperelliptic curve in terms of theta constants. The
theta null values are computed with respect to a semi-canonical theta
structure of level $2^\nu p$ where $\nu
>0$ is an integer and $p=\mathrm{char}(\F_q)>2$. The
results of this paper suggest a global positive answer to the question
whether there exists a quasi-quadratic time algorithm for the
computation of the number of rational points on a generic ordinary abelian
variety defined over a finite field.
\end{abstract}

\section{Introduction}
\label{intro}

The study of the properties of non-singular projective algebraic
curves over finite fields is a subject of central importance
in algorithmic number theory and cryptography.
It is well established that the Jacobian varieties of such curves
constitute a suitable family of groups to be used
in cryptographic protocols which are based upon the difficulty of solving
the \emph{discrete logarithm problem}. In order to avoid 'weak' Jacobians, i.e.
Jacobian varieties which
give a trivial instance of the general discrete logarithm problem, it
is necessary to precompute the number of rational points on a given
Jacobian.  This issue has prompted a lot of research, focused on the
design of efficient point counting algorithms.

Next we briefly recall how one can count points by
computing the eigenvalues of the absolute Frobenius
endomorphism on a Jacobian variety. We denote by $\F_q$ a
finite field with $q$ elements. Let
$\Sigma$ be the $q$-th power Frobenius morphism acting on the
algebraic closure $\overline{\F}_q$ of $\F_q$.  Let $C$ be a
smooth projective curve of genus $g$ over $\F_q$,
and let $J(C)$ be its Jacobian. For a prime number $\ell$ not dividing $q$
we denote by $T_\ell$ the $\ell$-adic Tate module of $J(C)$. The latter is a free
$\Z_\ell$-module of rank $2g$. Here $\Z_\ell$ stands for the
$\ell$-adic integers. Let $\End(J(C))$ be the ring of
endomorphisms of $J(C)$ and put
$\End^0(J(C))=\End(J(C))\otimes \Q$.
There exists a canonical injective morphism
$\rho_{\ell}:
\End^0(J(C)) \rightarrow
\End_{\Q_\ell} ( T_\ell \otimes \Q_\ell )$ which is called the
\emph{$\ell$-adic representation} of $\End^0(J(C))$. Let $F$ be the
purely inseparable endomorphism of degree $q^g$ of $J(C)$ given by the
action of $\Sigma$ on geometric point coordinates $(x_1,\ldots, x_n)
\mapsto (x_1^q,\ldots, x_n^q)$. One would like to compute, in an
efficient way, the characteristic polynomial $\chi_F$ of
$\rho_\ell(F)$.  One recovers the number of rational points of the
Jacobian $J(C)$ as $\chi_F(1)$.

Broadly speaking, there exists two classes of point counting
algorithms. On one hand, there are the so-called \emph{$\ell$-adic
  algorithms} initiated by the work of R. Schoof~\cite{Schoof2}. These
algorithms compute the action of the Frobenius morphism on the group
of $\ell$-torsion points for different primes $\ell$, where the latter
$\ell$ are chosen coprime to the characteristic of the finite field.
If the product over all $\ell$ is sufficiently big, then one can
recover $\chi_F$ by the Chinese remainder theorem. Schoof's algorithm
for elliptic curves behaves very well, due to the improvements by O.
Atkin and N. Elkies~\cite{Schoof1}~\cite{elk98}. Cryptographic sizes
still seem to be difficult to reach in genus $2$~\cite{GaSc04} and
higher.  A generalization of the method of R. Schoof, the complexity
of which is polynomial in the genus, has been proposed by B.
Edixhoven~\cite{edix06}.  On the other hand, there are the so-called
\emph{$p$-adic methods}, introduced by the work of T.
Satoh~\cite{MR2001j:11049}. These algorithms rely on the computation
of the action of the Frobenius morphism on $p$-adic canonical lifts of
certain arithmetic invariants, where $p>0$ is the characteristic of
the finite base field. They have in common a bad behavior with respect
to the characteristic $p$.  It is convenient to assess their
complexity in terms of $\log(q)$, where $q$ is the number of
elements of the finite field $\F_q$ and where the characteristic $p$
of the finite field is assumed as fixed.  \newline\indent In the
following we recall existing work about $p$-adic point counting
algorithms. First, a series of algorithmic improvements upon the
algorithm of Satoh~\cite{Gaudry2002, Harley2002, Harley2002a,
  KiPaChPaKiHa2002, LL03, MR1895422} led to a \emph{quasi-quadratic
  time point counting algorithm} for ordinary elliptic curves over
finite fields.  The special case of characteristic $2$ was then
interpreted by J.-F. Mestre in terms of a $2$-adic analogue of Gauss'
\emph{algebraic geometric mean}. He gave a very elegant and simple
quasi-cubic time point counting algorithm for ordinary elliptic curves
over finite fields of characteristic $2$~\cite{Mestre4}. The
previously cited algorithmic improvements upon the algorithm of Satoh
can also be applied to Mestre's algorithm, which results in a
quasi-quadratic time point counting algorithm. Mestre has extended the
scope of his algorithm by showing that the algebraic geometric mean
formulas can be considered as a particular case of the Riemann
duplication formulas for complex analytic theta
functions~\cite{Mestre5}. His ideas led to a quasi-quadratic time
point counting algorithm for ordinary hyperelliptic curves defined
over a finite field of characteristic $2$~\cite{LL07}.  Other $p$-adic
algorithms were found by K. Kedlaya~\cite{Kedlaya2001} and A. Lauder
\cite{MR2003g:11140}. Their algorithms are based on the computation of
the action of a formal Frobenius lift on the Monsky-Washnitzer and the
Dwork cohomology groups, respectively.  \newline\indent The aim of
this paper is to describe an algorithm for the computation of the
number of points of a generic ordinary hyperelliptic curve over a
finite field $\F_q$ with $q$ elements of characteristic $p>2$ which
has quasi-quadratic time and quadratic space complexity in terms of
$\log_p(q)$. We give two versions of our algorithm: a proven version
of the algorithm, for which we are able to prove that it is correct
and that it has quasi-quadratic time and quadratic space complexity,
and a heuristic version of the algorithm, the proof of which relies on
some yet unproven facts.

The reason why we give both algorithms is that, due to
the smaller constant term in the complexity estimate of the
heuristic algorithm,
it performs much faster than the proven algorithm for field sizes
which are actually used in the applications. We have strong computational
evidence that also the heuristic version of the algorithm is correct.
Our method follows the
point counting strategy of J.-F. Mestre, which relies on the
computation of arithmetic invariants of canonical lifts using the
coordinate system provided by the theta null values associated to an
abelian variety with theta structure.  In our case, the theta null
point is computed with respect to a theta structure of level $2^\nu p$
where $\nu >1$ is an integer and $p>2$ is the characteristic of the
residue field.  The results of this paper suggest a global positive
answer to the question whether there exists a quasi-quadratic time and
quadratic space algorithm for the computation of the number of
rational points of a generic ordinary abelian variety over a
finite field.

Both versions of the algorithm consist of the following two main
steps, according to the classical lift and norm paradigm. Let $C$ be
an ordinary hyperelliptic curve over a finite field of characteristic
$p>2$ whose Jacobian is absolutely simple.
\begin{enumerate}
\item First, one computes a certain arithmetic invariant associated to
  the canonical lift of the Jacobian variety of the curve $C$.  The
  arithmetic invariant is given by the theta null point of a Jacobian
  of $C$ with respect to a semi-canonical
  theta structure of level $2^\nu p$ with $\nu > 0$ an integer. The
  lifting is done using a multivariate Henselian lifting algorithm
  applied to certain theta identities of level
  $2^\nu p$ and degree $p^2$.
\item Secondly, the norm of a certain quotient of theta null values
  attached to the canonical lift is
  computed. This value coincides with the product of
  the invertible eigenvalues of the absolute Frobenius endomorphism on
  the reduction.  If one computes the canonical lift and the norm with
  sufficiently high precision, then it is straight forward to recover
  the characteristic polynomial of the Frobenius morphism from the
  latter approximation of the norm.
\end{enumerate}
The only difference between the proven and the heuristic version of
our algorithm lies in the choice of the parameter $\nu$ of Step 1.
In the heuristic version of the algorithm the parameter $\nu$
is chosen to be equal $1$.
In the case that $\nu = 3$
we are able to give a complete proof of correctness of
the algorithm.

This paper is organized as follows.  In Section \ref{thetaidentities}
we present some new theoretical results that form the basis of our algorithm.
\begin{itemize}
\item (Section \ref{sigmasquare})
An important ingredient of our algorithm is given by theta relations
of level $2^\nu p$ and degree $p^2$,
which describe the action of a square of the unique Frobenius lift on the
Serre-Tate formal torus with respect to the coordinates given by the
canonical theta structure~\cite{ca05a}.
We remark, that the equations, which are described in Section
\ref{sigmasquare}, can also be used for CM
construction in arbitrary characteristic generalizing
the results of~\cite{ckl08}.
\item (Section \ref{riemanntwop}) We give equations which, together
  with the relations of Section \ref{sigmasquare}, define the local
  deformation space of an ordinary abelian variety with a $(2^\nu
  p)$-theta structure. Classically, equations in terms of theta
  constants defining the moduli space of abelian varieties are known
  if the level is divisible by $8$ (see~\cite[$\S$6]{MR36:2621}).

\item (Section \ref{extheta2p})
It is well-known that the $4$-theta null point of the Jacobian variety
of a hyperelliptic
curve can be computed using the Thomae formulas. One can extend the
$4$-theta null point to a $(2^\nu p)$-theta null point using the equations
for level $2^{\nu}p$ that are given in Section \ref{riemanntwop}.
\item (Section \ref{gtrace})
We give a transformation formula which relates a certain quotient
of theta null values of the canonical lift
for level $2^{\nu}p$
with the product of the invertible eigenvalues of the absolute
Frobenius morphism acting on the reduction.

\end{itemize}
In Section \ref{algo} we give a point counting algorithm for generic
ordinary hyperelliptic curves over a finite field of characteristic
$p>2$.  In Section \ref{complexity} we provide a detailed complexity
analysis of the latter algorithm. In Section \ref{secfinit}, we prove
that a closed variety defined from the equations of Section
\ref{riemanntwop} has dimension $0$.  In Section \ref{implementation}
we give some examples that have been computed using an experimental
implementation of our algorithm.

\paragraph{Notations and complexity hypothesis.}

We will denote by $\xf_q$ a finite field of characteristic $p>0$
having $q$ elements. Let $\xz_q$ denote the ring of Witt vectors with
values in $\xf_q$ and by $\Q_q$ the field of fractions of $\Z_q$.
There exists a canonical lift $\sigma \in \mathrm{Aut}(\xz_q)$ of the
$p$-th power Frobenius morphism of $\xf_q$.  If $a$ is an element of
$\xz_q$ then we denote by $\bar{a}$ its reduction modulo $p$ in $\xf_q$.
We say that we have computed an element $x \in \xz_q$ to precision $m$,
if we we can write down a bit-string representing its class in the
quotient ring $\xz_q \slash p^m \xz_q$.  In order to assess the
complexity of our algorithm we use the computational model of a
Random Access Machine~\cite{MR1251285}. We assume that the
multiplication of two $n$-bit length integers takes $O(n^\mu)$ bit
operations.  One can take $\mu=1+\epsilon$ (for $n$ large),
$\mu=\log_2 (3)$ and $\mu=2$ using the FFT multiplication algorithm,
the Karatsuba algorithm and a naive multiplication method,
respectively.  Let $x,y \in \xz_q \slash p^m \xz_q$. For the following
we assume the sparse modulus representation which is explained in
\cite[pp.239]{MR2162716}.  Under this assumption one can compute the
product $xy$ to precision $m$ by performing $O(\log (q)^\mu m^\mu)$
bit operations.

\section{Theta relations of level $2p$}
\label{thetaidentities}

In this section we give some original results that form the basis of
our point counting algorithm.  In order to do explicit canonical
lifting it is necessary to find theta identities that describe the
\emph{arithmetic invariants of canonical lifts}.  It is not difficult
to make up a theta relation. A hard problem is to make a 'complete'
set of theta relations that is suitable for canonical lifting.  We
give such a complete set of equations in the following sections.  Also
we give a special theta relation, deduced from the classical
\emph{transformation formula}, which allows one to recover the
eigenvalues of the Frobenius from the arithmetic invariant of the
canonical lift.

\subsection{A local $p^2$-correspondence}
\label{sigmasquare}

Let $R$ be a complete noetherian local ring with
finite residue field $\F_q$ of characteristic $p>0$.
Suppose that we are given an abelian scheme $A$
over $R$ which has ordinary reduction.
Let $\pol$ be an ample symmetric line bundle of degree $1$ on $A$.
Assume that there exists a $\sigma \in \mathrm{Aut}(R)$ lifting the $p$-th power
Frobenius automorphism of $\F_q$.
For $m \geq 1$ we set $Z_m=(\xz / m\xz)^g$ where $g$ is the relative
dimension of $A$ over $R$.
\newline\indent
Now assume that $p>2$ and let $n \geq 1$ an integer with $(n,p)=1$,
i.e. $n$ is coprime to $p$.
Suppose that we are given a symmetric
theta structure $\Theta_{2n}$ of type $Z_{2n}$
for $\pol^{2n}$ and an isomorphism
\begin{eqnarray}
\label{trivp}
Z_{p,R} \stackrel{\sim}{\rightarrow} A[p]^{\mathrm{et}},
\end{eqnarray}
where $A[p]^{\mathrm{et}}$ denotes the maximal \'{e}tale
quotient of $A[p]$.
By~\cite[Th.2.2]{ca05a} there exists a canonical theta structure
$\Theta_p$ of type $Z_p$ for the line bundle $\pol^p$ which is uniquely
determined by the isomorphism (\ref{trivp}).
Let $\Theta_{2np}=\Theta_{2n} \times
\Theta_p$ be the semi-canonical symmetric product theta structure
of type $Z_{2np}$ for $\pol^{2np}$ (see~\cite[$\S$3.2]{ckl08}).
\newline\indent
We denote the theta null point
with respect to the theta structure $\Theta_{2np}$
by $(a_u)_{u \in Z_{2np}}$.
In the following we consider $Z_2$, $Z_{np}$ and
$Z_{2p}$ as embedded compatibly into $Z_{2np}$.
Let $S$ be the set of all $4$-tuples $(x,y,v,w) \in Z_{2np}^4$ such
that the sets $\{ x+y, x-y \}$ and $\{ v+pw, v-pw \}$
are equal and contained in $Z_{np}$.
\begin{theorem}
\label{thetarelations}
There exists an $\omega \in R^*$ such that for all $(x,y,v,w) \in S$
one has
\[
\sum_{z \in Z_2} a_{x+z} a_{y+z}
= \omega \sum_{u \in Z_{2p}}
a_{v+pu} a^{\sigma^2}_{w+u}.
\]
\end{theorem}
\begin{proof}
Assume that we have chosen an isomorphism
\begin{eqnarray}
\label{trivp3}
Z_{p^3,R} \stackrel{\sim}{\rightarrow}
A[p^3]^{\mathrm{et}}
\end{eqnarray}
which induces the trivialization (\ref{trivp}) if one restricts to $Z_p$.
The choice of the isomorphism (\ref{trivp3}) possibly requires
a local-\'{e}tale extension of the base.
Nevertheless, the resulting
formulas are defined over the original ring $R$.
By~\cite[Th.2.2]{ca05a} there exists a canonical theta structure
$\Theta_{p^3}$ of type $Z_{p^3}$ for the line bundle $\pol^{p^3}$
depending on the trivialization (\ref{trivp3}).
By~\cite[Lem.2.1]{ca05b} the theta structures $\Theta_{p^3}$ and
$\Theta_p$ are $p^2$-compatible in the sense of~\cite[Def.5.5]{ca05b}.
By~\cite[Lem.3.3]{ckl08} there exists a semi-canonical product theta
structure $\Theta_{2np^3}=\Theta_{2n} \times \Theta_{p^3}$
of type $Z_{2np^3}$ for $\pol^{2np^3}$.
We remark that by~\cite[Th.5.1]{ca05a} and~\cite[Lem.3.2]{ckl08}
the canonical theta structure $\Theta_{p^3}$ is symmetric.
Hence by~\cite[Lem.3.4]{ckl08} the theta structures
$\Theta_p$, $\Theta_{2np}$, $\Theta_{p^3}$
and $\Theta_{2np^3}$ form a compatible system.
Because of the symmetry of the theta structure $\Theta_{2n}$ there
exists a theta structure $\Theta_n$ of type
$Z_n$ for $\pol^n$ which is $2$-compatible with $\Theta_{2n}$ (see~\cite[$\S$2,Rem.1]{MR34:4269}). By the same reasoning as above there exists a semi-canonical product theta structure $\Theta_{np}=\Theta_n \times \Theta_p$ which is $2$-compatible with $\Theta_{2np}$.
\newline\indent
Suppose that we are given a rigidification of the line bundle $\pol$.
We set $V(Z_m)=\underline{\mathrm{Hom}}(Z_{m,R}, \mathcal{O}_R)$ for $m \geq 1$. 
Recall that $V(Z_m)$ is the module of finite theta functions
as defined in~\cite[$\S$1]{MR34:4269}.
One can choose theta group equivariant isomorphisms
\[
\mu_i: \pi_* \pol^i \stackrel{\sim}{\rightarrow} V(Z_i),
\]
where $i \in I= \{ np,2np,2np^3 \}$ and where $\pi:A \rightarrow \spec{R}$
denotes the structure morphism.
The isomorphisms $\mu_i$ induce finite theta functions $q_{\pol^i} \in
V(Z_i)$ where $i \in I$ . 
\newline\indent
It follows from Corollary \ref{mummult} taking
$i=j=1$ and $m=-n=p$ that there exists a $\lambda
\in R^*$ such that
\[
q_{\pol^{np}}(v+pw) q_{\pol^{np}}(v-pw) = \lambda
\sum_{u \in Z_{2p}} q_{\pol^{2np}}(v+pu) q_{\pol^{2np^3}}(w+u).
\]
It follows by~\cite[Th.2.4]{ckl08} and~\cite[Th2.3]{ca05b}
in conjunction with
\cite[Lem2.2]{ca05b} and
\cite[Lem.3.5]{ckl08} that
\begin{eqnarray}
\label{key1}
q_{\pol^{np}}(v+pw) q_{\pol^{np}}(v-pw) = \lambda
\sum_{u \in Z_{2p}} q_{\pol^{2np}}(v+pu) q_{\pol^{2np}}(w+u)^{\sigma^2}.
\end{eqnarray}
Corollary \ref{mummult} implies by means
of the choice $i=j=p$ and $m=-n=1$
that there exists a $\lambda \in R^*$ such that
\begin{eqnarray}
\label{key2}
q_{\pol^{np}}(x+y) q_{\pol^{np}}(x-y) = \lambda
\sum_{z \in Z_{2}} q_{\pol^{2np}}(x+z) q_{\pol^{2np}}(y+z).
\end{eqnarray}
By the assumption that $(x,y,v,w) \in S$,
the left hand sides of the equations (\ref{key1}) and
(\ref{key2}) are equal. As a consequence, there exists an $\omega \in R^*$ such
that
\[
\sum_{z \in Z_{2}} q_{\pol^{2np}}(x+z) q_{\pol^{2np}}(y+z) = \omega
\sum_{u \in Z_{2p}} q_{\pol^{2np}}(v+pu) q_{\pol^{2np}}(w+u)^{\sigma^2}.
\]
This completes the proof of the theorem.
\end{proof}
\noindent
In the following we illustrate Theorem \ref{thetarelations} by some examples.
\newline\newline
\bfseries Example $g=1$, $n=1$, $p=3$: \mdseries
\begin{eqnarray*}
a_{1}a_{0}+a_{2}a_{3} & = & \omega (a_{1}a^{\sigma^2}_{0}+a_{2}a^{\sigma^2}_{3}+2a_{1}a^{\sigma^2}_{2}+2a_{2}a^{\sigma^2}_{1}) \\
a_{2}a_{0}+a_{1}a_{3} & = & \omega (2a_{1}a^{\sigma^2}_{1}+a_{2}a^{\sigma^2}_{0}+a_{1}a^{\sigma^2}_{3}+2a_{2}a^{\sigma^2}_{2}) \\
a_{3}a_{3}+a_{0}a_{0} & = & \omega (2a_{0}a^{\sigma^2}_{2}+a_{3}a^{\sigma^2}_{3}+2a_{3}a^{\sigma^2}_{1}+a_{0}a^{\sigma^2}_{0}) \\
a_{0}a_{3}+a_{3}a_{0} & = & \omega (a_{0}a^{\sigma^2}_{3}+2a_{3}a^{\sigma^2}_{2}+a_{3}a^{\sigma^2}_{0}+2a_{0}a^{\sigma^2}_{1})
\end{eqnarray*}
\bfseries Example $g=1$, $n=1$, $p=5$: \mdseries
\begin{eqnarray*}
a_{2}a_{0}+a_{3}a_{5} & = & \omega
(2a_{3}a^{\sigma^2}_{3}+2a_{3}a^{\sigma^2}_{1}+a_{2}a^{\sigma^2}_{0}+2a_{2}a^{\sigma^2}_{4}
+2a_{2}a^{\sigma^2}_{2}+a_{3}a^{\sigma^2}_{5}) \\
a_{2}a_{5}+a_{3}a_{0} & = & \omega
(2a_{2}a^{\sigma^2}_{3}+a_{2}a^{\sigma^2}_{5}+2a_{3}a^{\sigma^2}_{4}+2a_{3}a^{\sigma^2}_{2}
+2a_{2}a^{\sigma^2}_{1}+a_{3}a^{\sigma^2}_{0}) \\
a_{0}a_{5}+a_{5}a_{0} & = & \omega
(a_{0}a^{\sigma^2}_{5}+2a_{5}a^{\sigma^2}_{4}+2a_{0}a^{\sigma^2}_{3}+2a_{5}a^{\sigma^2}_{2}
+a_{5}a^{\sigma^2}_{0}+2a_{0}a^{\sigma^2}_{1}) \\
a_{4}a_{0}+a_{1}a_{5} & = & \omega
(a_{4}a^{\sigma^2}_{0}+2a_{1}a^{\sigma^2}_{1}+a_{1}a^{\sigma^2}_{5}+2a_{1}a^{\sigma^2}_{3}
+2a_{4}a^{\sigma^2}_{4}+2a_{4}a^{\sigma^2}_{2}) \\
a_{5}a_{5}+a_{0}a_{0} & = & \omega
(2a_{0}a^{\sigma^2}_{2}+2a_{5}a^{\sigma^2}_{1}+2a_{0}a^{\sigma^2}_{4}+a_{5}a^{\sigma^2}_{5}
+2a_{5}a^{\sigma^2}_{3}+a_{0}a^{\sigma^2}_{0}) \\
a_{1}a_{0}+a_{4}a_{5} & = & \omega
(a_{1}a^{\sigma^2}_{0}+2a_{4}a^{\sigma^2}_{3}+a_{4}a^{\sigma^2}_{5}+2a_{1}a^{\sigma^2}_{4}
+2a_{1}a^{\sigma^2}_{2}+2a_{4}a^{\sigma^2}_{1})
\end{eqnarray*}

\subsubsection{A generalized theta multiplication formula}
\label{genmult}

In the following we give a generalized
multiplication formula in the context
of Mumford's algebraic theta functions.
We only sketch a proof. For more details
we refer to~\cite{ko76} and~\cite{ke89}. 
\newline\indent
Let $A$ be an abelian scheme over a local ring $R$ and let
$\xi$ denote the isogeny $A^2 \rightarrow A^2$ given by the
matrix
\[
\left(
\begin{array}{cc}
1 & m \\
1 & n
\end{array}
\right)
\]
where $m,n \in \xz$.
Let $i,j \geq 1$ and $I= \{ i, j , i+j , i m^2+ j n^2 \}$.
Assume that we are given an ample symmetric line bundle $\pol$
on $A$ and compatible theta structures
$\Theta_i$ for $\pol^i$ of type $K_i$ where $i \in I$.
We set $\emm_{i,j}=p_1^* \pol^i \otimes p_2^* \pol^j$.
\begin{lemma}
\label{isso}
Suppose that $i m + j n = 0$.
Then one has
\[
\xi^* \emm_{i,j} \cong \emm_{i+j,i m^2+ j n^2}.
\]
\end{lemma}
\begin{proof}
Let $(a,b) \in A^2$. We define
\[
s_1:A \rightarrow A^2,x \mapsto (a,x)
\quad
\mbox{and} \quad
s_2:A \rightarrow A^2,x \mapsto (x,b).
\]
One computes
\begin{eqnarray*}
s_2^* \emm_{i,j} = s_2^* p_1^* \pol^i \otimes s_2^* p_2^* \pol^j
 = (p_1 \circ s_2)^* \pol^i \otimes (p_2 \circ s_2)^* \pol^j = \pol^i
\end{eqnarray*}
where $p_k:A^3 \rightarrow A$ denotes the projection on the $k$-th
factor.
Similarly, we have $s_1^* \emm_{i,j} = \pol^j$.
Also we have
\begin{eqnarray*}
& s_2^* \xi^* \emm_{i,j} & = (p_1 \circ \xi \circ s_2)^* \pol^i
\otimes (p_2 \circ \xi \circ s_2)^* \pol^j   =  T_{[m]b}^* \pol^i \otimes T_{[n]b}^* \pol^j \\
& & \stackrel{(*)}{=}  \big( T_b^* \pol^{im} \otimes \pol^{-i(m-1)} \big) \otimes \big(
T_b^* \pol^{jn} \otimes \pol^{-j(n-1)} \big) \\
& & = T_b^* \pol^{im+jn} \otimes \pol^{-(im+jn)+(i+j)}  = \pol^{i+j}.  
\end{eqnarray*}
The latter equality follows by our assumption $im+jn=0$.
The equality $(*)$ is implied by the \emph{Theorem of the Square}.
Now take $a=0_A$ where $0_A$ denotes the zero section of $A$.
Then one has
\begin{eqnarray*}
\lefteqn{ s_1^* \xi^* \emm_{i,j} 
= (p_1 \circ \xi \circ s_1)^* \pol^i
\otimes (p_2 \circ \xi \circ s_1)^* \pol^j} \\
& & = [m]^* \pol^i
\otimes [n]^* \pol^j = \pol^{im^2+jn^2}.
\end{eqnarray*}
The latter equality comes from the symmetry of the line bundle $\pol$.
The proposition now follows by applying the \emph{Seesaw Principle}.
\end{proof}
\noindent
Assume now that we are given $n,m \in \xz$ such that $i m + j n = 0$.
There exists a product theta
structure $\Theta_{i,j}$ of type $K_{i,j}$ for $\emm_{i,j}$ where
$K_{i,j}=K_i \times K_j$.
On top of Lemma \ref{isso} one can verify that the theta structure
$\Theta_{i+j,i m^2+ j n^2}$ is $\xi$-compatible with
the theta structure $\Theta_{i,j}$
(compare~\cite[$\S$3]{MR34:4269} and~\cite[Lem.3.8]{ckl08}).
Hence we can apply the Isogeny Theorem
(see~\cite[$\S$1,Th.4]{MR34:4269}) in order
to get the following general addition formula.
\begin{proposition}
\label{generaladd}
There exists a $\lambda \in R^*$ such that for all $g \in
V(K_{i,j})$ and $(x,y) \in K_{i+j,i m^2+ j n^2}$ we have
\[
\xi^*(g)(x,y)= \left\{
\begin{array}{l@{, \quad}l}
\lambda g \big( \xi(x,y) \big) & \xi(x,y) \in K_{i,j} \\
0 & \mathrm{else}
\end{array} \right.
\]
\end{proposition}
\noindent
Here we denote by $V(K_{i,j})$ the module of finite
theta functions of type $K_{i,j}$.
We define for $x \in K_{i+j}$
\[
G_x = \{ y \in K_{i m^2+ j n^2} | \xi(x,y) \in K_{i,j} \}.
\]
Here $K_{i m^2+ j n^2}$ and $K_{i,j}$ are considered as subgroups
of $A$ and $A^2$ via the theta structures
$\Theta_{i+j,i m^2+ j n^2}$ and $\Theta_{i,j}$, respectively.
As a corollary of Proposition \ref{generaladd} we get the following theorem.
\begin{theorem}[General Multiplication Formula]
\label{generalmult}
There exists a $\lambda \in R^*$ such that for all $x \in K_{i+j}$,
$f_1 \in V(K_i)$ and $f_2 \in V(K_j)$ we have
\[
(f_1 \star f_2)(x)= \lambda  \sum_{ y \in G_x} f_1(x+ m y) f_2(x+
n y) q_{ \pol^{i m^2+ j n^2}}(y).
\]
\end{theorem}
\noindent
The $\star$-product is defined as in~\cite[$\S$3]{MR34:4269}.
A proof of Theorem \ref{generalmult} in terms of the classical analytic theory
is given in~\cite{ko76}.
In~\cite{ke89} the author sketches
a proof of the general multiplication formula over a field of positive characteristic.
We remark that for $i=j=m=-n=1$ one obtains Mumford's
$2$-multiplication formula~\cite[$\S$3]{MR34:4269}.
\begin{corollary}
\label{mummult}
There exists a $\lambda \in R^*$ such that for all $(a,b) \in K_{i,j}$ we
have
\[
q_{\pol^i}(a) q_{\pol^j}(b) = \lambda \sum_{\xi(x,y)=(a,b)}
q_{\pol^{i+j}}(x) q_{\pol^{i m^2 + j n^2} }(y).
\]
\end{corollary}

\subsection{Riemann's equations for level $2^\nu p$}
\label{riemanntwop}

We use the notation that has been introduced in
Section \ref{sigmasquare}.
Let $R$ be a noetherian local ring, $\ell>0$ a prime and
$\nu \geq 1$ an integer.
Suppose we are given an abelian scheme $A$ of relative
dimension $g$ over $R$.
Assume that we are given an ample symmetric line bundle $\pol$ of degree
$1$ on $A$ and a symmetric theta structure of type $Z_{2^\nu \ell}$
for the line bundle $\pol^{2^\nu \ell}$ where $Z_{2^\nu \ell}$ is as in Section
\ref{sigmasquare}.
We denote the theta null point
with respect to the theta structure $\Theta_{2^\nu \ell}$
by $(a_u)_{u \in Z_{2^\nu \ell}}$.
By symmetry we have $a_u = a_{-u}$ for all $u \in Z_{2^\nu \ell}$.
\newline\indent
The higher dimensional analogue of \emph{Riemann's equation}
for the case of a level-$2^\nu \ell$ theta structure is given by the
following theorem.
We consider quadruples $(v_i,w_i,x_i,y_i) \in Z_{2^{\nu}\ell}^4$ where $i=1,2$ as
equivalent if there exists a permutation matrix $P \in \mat{4}{\xz}$ such that
\[
(v_1+w_1,v_1-w_1,x_1+y_1,x_1-y_1)=(v_2+w_2,v_2-w_2,x_2+y_2,x_2-y_2)P.
\]
Let $\hat{Z}_2$ be the character group of $Z_{2}$.
\begin{theorem}
\label{riemannrelations}
For equivalent quadruples $(v_1,w_1,x_1,y_1),(v_2,w_2,x_2,y_2) \in
Z_{2^{\nu} \ell}^4$ and for all $\chi \in \hat{Z}_2$ the following equality holds
\begin{eqnarray*}
\lefteqn{ \sum_{t \in Z_2} \chi(t) a_{v_1+t} a_{w_1+t} \sum_{s \in Z_2} \chi(s)a_{x_1+s} a_{y_1+s} } \\
& & =\sum_{t \in Z_2} \chi(t) a_{v_2+t} a_{w_2+t} \sum_{s \in Z_2} \chi(s) a_{x_2+s} a_{y_2+s}.
\end{eqnarray*}
\end{theorem}
\noindent
We refer to~\cite[$\S$3]{MR34:4269} for a proof of this theorem.
\newline\newline
\bfseries Example $g=1$, $p=3$, $\nu=1$: \mdseries
\begin{eqnarray*}
& 0= &a_{1}a_{0}^2a_{3}-2a_{1}^2a_{2}^2+a_{2}a_{0}a_{3}^2\\
& 0= &a_{2}a_{0}^3+a_{1}a_{0}^2a_{3}-a_{2}^4-2a_{1}^2a_{2}^2-a_{1}^4+a_{1}a_{3}^3+a_{2}a_{0}a_{3}^2
\end{eqnarray*}
\bfseries Example $g=1$, $p=5$, $\nu=1$: \mdseries
\begin{eqnarray*}
& 0= &-a_{5}^2a_{2}a_{4}+a_{2}^2a_{4}^2+a_{3}^2a_{4}^2-a_{1}a_{0}^2a_{3}+a_{1}^2a_{2}^2-a_{5}^2a_{1}a_{3}+a_{1}^2a_{3}^2-a_{2}a_{0}^2a_{4}\\
& 0= &-a_{1}^2a_{0}a_{4}+a_{2}a_{3}^2a_{4}-a_{5}a_{1}a_{4}^2+a_{1}a_{2}^2a_{3}\\
& 0= &-a_{5}a_{0}a_{3}a_{4}+2a_{1}a_{2}a_{3}a_{4}-a_{5}a_{1}a_{2}a_{0}\\
& 0= &-a_{5}^2a_{2}a_{0}+2a_{1}^2a_{4}^2-a_{5}a_{0}^2a_{3}\\
& 0= &a_{2}^3a_{0}-a_{1}^3a_{3}+a_{5}a_{3}^3+a_{2}a_{0}a_{3}^2+a_{5}a_{2}^2a_{3}-a_{1}a_{3}a_{4}^2-a_{2}a_{4}^3-a_{1}^2a_{2}a_{4}\\
& 0= &-2a_{1}a_{2}a_{3}a_{4}+a_{5}^2a_{1}a_{3}-a_{1}^2a_{3}^2+a_{5}a_{1}a_{2}a_{0}+a_{2}a_{0}^2a_{4}-a_{2}^2a_{4}^2+a_{5}a_{0}a_{3}a_{4}\\
& 0= &-a_{5}^2a_{2}a_{4}+a_{3}^2a_{4}^2-a_{5}a_{0}a_{3}a_{4}+a_{1}^2a_{2}^2+2a_{1}a_{2}a_{3}a_{4}-a_{5}a_{1}a_{2}a_{0}-a_{1}a_{0}^2a_{3}\\
& 0= &a_{5}^2a_{0}a_{4}-2a_{2}^2a_{3}^2+a_{5}a_{1}a_{0}^2\\
& 0= &a_{2}a_{0}a_{3}^2+a_{5}a_{2}^2a_{3}-a_{1}a_{3}a_{4}^2-a_{1}^2a_{2}a_{4}\\
& 0= &-a_{1}^2a_{0}a_{4}+a_{2}^3a_{4}+a_{1}a_{3}^3-a_{0}a_{4}^3-a_{5}a_{1}^3-a_{5}a_{1}a_{4}^2+a_{2}a_{3}^2a_{4}+a_{1}a_{2}^2a_{3}\\
& 0= &a_{2}a_{0}^3-a_{1}^4+a_{5}a_{0}^2a_{3}+a_{5}^3a_{3}+a_{5}^2a_{2}a_{0}-2a_{1}^2a_{4}^2-a_{4}^4\\
& 0= &a_{5}^3a_{1}-a_{2}^4+a_{0}^3a_{4}+a_{5}^2a_{0}a_{4}-a_{3}^4-2a_{2}^2a_{3}^2+a_{5}a_{1}a_{0}^2
\end{eqnarray*}

\subsection{Theta null points of level $2^\nu p$}
\label{extheta2p}

Let $\F_q$ be a finite field of characteristic $p>2$. Let $A_{\F_q}$
be an ordinary abelian variety over $\F_q$. Suppose that we are given
a semi-canonical symmetric product theta structure $\Theta_{2^\nu
  p}=\Theta_{2^\nu} \times \Theta_p$ as in Section \ref{sigmasquare}.
We denote the theta null point with respect to the theta structure
$\Theta_{2^\nu }$ by $(a_u)_{u \in Z_{2^\nu }}$.  We can assume that
there exists a $v \in Z_{2^\nu}$ such that $a_v$ is a unit in $\Z_q$.
Here $Z_{2^\nu}$ is considered as a subgroup of $Z_{2^\nu p}$ via the
map $j \mapsto pj$.  Let $I$ be the ideal of the multivariate
polynomial ring $\F_q[x_u|u \in Z_{2^\nu p}]$ which is spanned by the
relations of Theorem \ref{riemannrelations}, taken modulo $p$,
together with the symmetry relations $a_u = a_{-u}$ for all $u \in
Z_{2^\nu p}$.  Let $J$ be the image of $I$ under the specialization
map
\begin{eqnarray*}
\F_q[x_u|u \in
Z_{2^\nu p}] \rightarrow \F_q[x_u|u \in
Z_{2^\nu p},2^\nu u \not=0], \quad x_u \mapsto \left\{
\begin{array}{l@{,\hsp}l}
\frac{a_u}{a_v} & \mathrm{if} \quad u \in Z_{2^\nu} \\
\frac{x_u}{a_v} & \mathrm{else}
\end{array} \right. .
\end{eqnarray*}
The following Theorem is proven in Section \ref{secfinit}.
\begin{theorem}\label{multiplicity}
If $\nu \geq 2$, then the ideal $J$ defines a $0$-dimensional affine
algebraic set. 
\end{theorem}
\noindent
By the primitive element theorem
there exists $f(x) \in \F_q[x]$ such that
\[
\F_q[x_u|u \in Z_{2^\nu p},2^\nu u \not=0]/\mathrm{rad}(J) \cong \F_q[x]/(f).
\]
The theta null point $(a_u)_{u \in Z_{2^\nu p}}$ induces an element $z
\in \F_q$ such that $f(z)=0$.
Generically, one can obtain the polynomial $f$ by a Groebner
basis computation.
The Theorem \ref{multiplicity} enables one to calculate the full theta
null point $(a_u)_{u \in Z_{2^\nu p}}$ over $\F_q$ from the knowledge
of its $2^\nu$-torsion part.  As a consequence, by means of the
well-known Thomae formulas and a Groebner basis computation algorithm,
one can produce arbitrary theta null points of level $2^\nu p$, which
correspond to ordinary hyperelliptic curves over $\F_q$.

We remark that in the case $\nu=1$, we have computationally verified
in many cases that the conclusion of Theorem \ref{multiplicity}
still holds.

\subsection{A generalized trace formula}
\label{gtrace}

Let $A$ be an abelian scheme over $\Z_q$. We assume that $A$ has
ordinary reduction and that it is the canonical lift of the reduction
$A_{\F_q}$.  Suppose that $\Theta_{2^\nu p} = \Theta_{2^\nu} \times
\Theta_p$ is a semi-canonical symmetric product theta structure over
$\Z_q$ of type $Z_{2^\nu p}$ for $\pol^{2^\nu p}$.  Let $(a_u)_{u \in
  Z_{2^\nu p}}$ denote the theta null point with respect to the theta
structure $\Theta_{2^\nu p}$.  \newline\indent Let $F \in
\mathrm{End}_{\F_q}(A_{\F_q})$ be the absolute Frobenius endomorphism
of $A_{\F_q}$, and let $\ell$ be a prime different from the
characteristic $p$ of $\F_q$.  We
denote the $\ell$-adic Tate module of $A_{\F_q}$ by
$T_\ell(A_{\F_q})$.  Recall that the $\ell$-adic Tate module is a free
$\Z_\ell$-module of rank $2g$, where $g$ is the dimension
of $A_{\F_q}$.  The absolute Frobenius morphism $F$
induces a $\Z_\ell$-linear map $\rho_\ell(F)$ on $T_\ell(A_{\F_q})$ which
corresponds, once a basis of $T_\ell(A_{\F_q})$ is chosen, to a $(2g
\times 2g)$-matrix $M_F$ with coefficients in $\Z_\ell$.  Because of
the ordinary reduction, we know that $M_F$ has precisely $g$ Eigenvalues
$\pi_1, \ldots , \pi_g$, which are units modulo $p$~\cite[Ch.V]{de72}.
\begin{theorem}\label{formula1}
  Suppose that $\Theta_{2^\nu}$ is defined over $\Z_q$. Then the
  product $\pi_1 \cdot \ldots \cdot \pi_g$ is an element of the ring
  $\Z_q$ and we have
\begin{equation}
  \pi_1 \cdot \ldots \cdot \pi_g = \Norm_{\Q_q / \Q_p} \left(
    \frac{\sum_{u \in Z_{2^\nu} }a_u}{\sum_{u \in Z_{2^\nu p}}a_u} \right).
\end{equation}
\end{theorem}
\noindent
Here $Z_{2^\nu}$ is considered as a subgroup of $Z_{2^\nu p}$
via the map $j \mapsto pj$.

The rest of this section is devoted to the proof of Theorem
\ref{formula1}.  We first fix some additional notations. If $\pol$ is
a line bundle on an abelian variety, we denote by $K(\pol)$ the kernel
of the isogeny $A \rightarrow \mathrm{Pic}^0_A$ induced by $\pol$. Denote by
$\gtheta(\pol)$ the theta group associated to $\pol$ (see~\cite[pp.
289]{MR34:4269}).  For any positive integer $n$, we denote the
Heisenberg group of type $Z_n$ by $\mathcal{H}(Z_n)$~\cite[pp.
161]{MR2062673}.  Denote by $\hat{Z}_n$ the dual of $Z_n$, we have by
definition $\mathcal{H}(Z_n)= \xg_m \times Z_n \times \hat{Z}_n$
together with the group law defined by
$$(\alpha, x, l). (\alpha', x', l')=(\alpha.\alpha' l'(\alpha), x+x', l.l').$$ 
where $(\alpha, x, l)$ and
$(\alpha', x', l')$ are points of $\mathcal{H}(Z_n)$.

During the course of the proof, as we are working with schemes over
different base rings, to avoid ambiguity, we recall the base ring in
subscript. In particular, we let $A=A_{\Z_q}$, $\pol = \pol_{\Z_q}$
and $\Theta_{2^\nu p}= \Theta_{\Z_q, 2^\nu p}$.  We recall that
$\Theta_{2^\nu p}$ induces a decomposition $K(\pol)=K_1(\pol^{2^\nu
  p}) \times K_2(\pol^{2^\nu p})$ into isotropic subgroups
$K_1(\pol^{2^\nu p})$ and $K_2(\pol^{2^\nu p})$ for the commutator
pairing.

We fix an embedding $\psi : \C_p\rightarrow \C$ where $\C_p$ is the
completion of the algebraic closure of $\Q_p$~\cite[Ch.3]{MR1760253}.
The base extended abelian variety $A_\C = A_{\Z_q} \times_\psi
Spec(\C)$ is a complex variety with a polarization $\pol^{2^\nu p}_\C$
defined by $\pol^{2^\nu p}_{\C}=\pol^{2^\nu p}_{\Z_q} \otimes_{\psi}
\C$. We remark that $K(\pol^{2^\nu p}_\C)$ comes equipped with a
Lagrangian decomposition which is inherited from the theta structure
$\Theta_{2^\nu p,\C}=\Theta_{2^\nu p} \otimes \C$.  From the above
decomposition we deduce the period matrix $(I \Omega)$ with $I$ the
$g$ dimensional unity matrix and $\Omega$ an element of $\Si_g$ the
$g$ dimensional Siegel upper half space. In the following, for any
$\Omega \in \Si_g$, we denote by $\Lambda_\Omega$ the lattice $\Z^g +
\Omega \Z^g$.  If we let $A_{an}=\C^g / \Lambda_\Omega$, we have an
analytic isomorphism $j_{an} : A_\C \rightarrow A_{an}$. Let $\kappa :
\C^g \rightarrow \C^g / \Lambda_{\Omega}$ be the canonical projection.

We can suppose that $\Omega$ is chosen such that the $p$-torsion
points of $A_{an}$, given by $\kappa((1/p).\Z^g)$ corresponds via
$j_{an}^{-1}$ to a canonical lift of the maximal \'etale quotient of
$A_{\Z_q}[p]$, where $A_{\Z_q}$ is identified to $A_{\C}$ via $\psi$.

For $\epsilon_1, \epsilon_2, l \in \Z$, we define the theta function with
rational characteristics as
\begin{equation}\label{thetafunction}
  \theta_l\carac \epsilon_1, {\epsilon_2} (z,\Omega) = \sum_{n \in \Z^g} \exp \Big[ \pi i ^t(n
  + \frac{\epsilon_1}{l}) \Omega (n+\frac{\epsilon_1}{l}) + 2 \pi i ^t
  (n+\frac{\epsilon_1}{l}).(z+\frac{\epsilon_2}{l}) \Big].
\end{equation}
Recall that $(a_u)_{u \in Z_{2^\nu p}}$ denote the theta null point
with respect to the theta structure $\Theta_{2^\nu p}$.  We have the
\begin{lemma}\label{lemma1}
  There exists a constant factor $\lambda \in \C$, $\chi \in
  \hat{Z}_{2^\nu p}$ a character of order $2$ and $\delta \in Z_2$,
  such that for all $u\in Z_{2^\nu p}$,
\begin{equation}\label{rel2}
  (a_{u}\otimes_{\Q_q} \C) = \lambda \chi(u)
  \theta_{2^\nu p}\carac 0, {u+\delta} (0,1/(2^\nu p).\Omega),
\end{equation}
where $Z_2$ is considered as a subgroup of $Z_{2^\nu p}$ via the map $j
\mapsto j 2^{\nu-1}p$.
\end{lemma}
\begin{proof}
  From the theta structure $\Theta_{\Z_q, 2^\nu p}$ of type $Z_{2^\nu
    p}$ we deduce immediately by tensoring with $\C$ a theta structure
  $\Theta_{\C,2^\nu p}$ of type $Z_{2^\nu p}$ for $\pol^{2^\nu p}_\C$.
  Then $(a_{u}\otimes_{\Q_q} \C)_{u \in Z_{2^\nu p}}$ is the theta
  null point defined by the theta structure $\Theta_{\C,2^\nu p}$. As
  $\pol^{2^\nu p}_\C$ is by hypothesis a symmetric line bundle, by
 ~\cite[Lem.4.6.2]{MR2062673}, there exists a $\overline{c} \in
  A_\C[2] \cap K(\pol^{2^\nu p}_\C)$ such that
  $\tau^*_{\overline{c}}(\pol^{2^\nu p}_\C) \simeq \pol^{2^\nu p}_0$,
  where $\pol^{2^\nu p}_0$ is the canonical bundle associate to the
  decomposition provided by the matrix period $\Omega$ (see
 ~\cite[Lem.3.1.1]{MR2062673}).
  
  The line bundle $\pol_0^{2^\nu p}$ comes with a symmetric theta
  structure $\Theta_0$ defined by the decomposition associated to
  $\Omega$ and the element $0 \in K(\pol^{2^\nu p}_0)$ (see
 ~\cite[Lem.6.6.5]{MR2062673}). The theta null point for the theta
  structure $\Theta_0$ is $$(\theta_{2^\nu p}\carac 0, {u} (0,1/(2^\nu
  p).\Omega))_{u \in Z_{2^\nu p}}$$ by~\cite[Prop.6.7.1]{MR2062673}.
  
  As $\overline{c} \in K(\pol^{2^\nu p}_\C)$, we have an isomorphism
  of theta groups $\zeta : \gtheta (\pol_\C^{2^\nu p}) \rightarrow
  \gtheta(\pol_0^{2^\nu p})$ defined by $\zeta ((y, \psi_y)) = (y,
  t^*_y \tau_{\overline{c}}^{*} \circ \psi_y \circ
  \tau_{\overline{c}}^{-1*})$ where $t_{y}$ denotes the
  translation by $y$. Note that this isomorphism induces the
  identity on $K(\pol_\C^{2^\nu p})=K(\pol_0^{2^\nu p})$.
  
  The isomorphism $\Theta_0 \circ \zeta: \gtheta(\pol_\C^{2^\nu p})
  \rightarrow \mathcal{H}(Z_{2^\nu p})$ is a theta structure for
  $\gtheta(\pol_\C^{2^\nu p})$. Denote by $\overline{\Theta}_0$ the
  morphism $K(\pol_\C^{2^\nu p}) \rightarrow Z_{2^\nu p} \times
  \hat{Z}_{2^\nu p}$ deduced from $\Theta_0$.  By definition of
  $\zeta$, the theta null point for the theta structure $\Theta_0
  \circ \zeta$ is deduced from the theta null point for $\Theta_0$ by
  acting upon it with $\overline{\Theta}_0(c)$ (for a definition of this
  action see \cite[pp.297]{MR34:4269}) so that it can be written as
  $(\chi_1(u) \theta\carac 0, {u+\delta_1} (z,1/(2^\nu p).\Omega))_{u
    \in Z_{2^\nu p}}$ where$\overline{\Theta}_0(c)=(\delta_1, \chi_1)
  \in Z_{2^\nu p} \times \hat{Z}_{2^\nu p}$.
 
  As $\Theta_{\C, 2^\nu p}$ and $\Theta_0 \circ \zeta$ are two
  symmetric theta structures of $\pol^{2^\nu p}_\C$ which induce the
  same symplectic isomorphism $\overline{\Theta}_0$, they are defined
  up to a translation by an element $c_0$ in $A_\C[2]$ by
 ~\cite[Prop.6.9.4]{MR2062673}.  Let
  $(\delta_2, \chi_2)=\overline{\Theta}_0(c_0)$, a theta null point for
  $\pol_\C^{2^\nu p}$ with the theta structure $\Theta_{\C, 2^\nu p}$
  is given modulo multiplication by a factor independent of $u$ by
$$(\chi_1(u) \chi_2(u) \theta_{2^\nu p}\carac 0,
{u +\delta_1+\delta_2} (z,1/(2^\nu p).\Omega))_{u \in Z_{2^\nu p}}.$$
We remark that $\chi_1, \chi_2$ and $\chi$ are charaters of order $2$
of $Z_{2^\nu p}$.  We conclude the proof by setting $\delta =
\delta_1+\delta_2$ and $\chi=\chi_1.\chi_2$.
\end{proof}

\begin{lemma}\label{lemma2}
  Let $F$ be the Frobenius morphism acting on $A_{\F_q}$ and let
  $\pi_1, \ldots , \pi_g$ be the Eigenvalues of the $\ell$-adic
  representation $\rho_\ell(F)$ which are units modulo $p$. Let
  $n=\log_p(q)$. For all $\epsilon_1, \epsilon_2 \in Z_2$,
  we have
$$
\frac{{\theta_2 \carac \epsilon_1, {\epsilon_2} (0, 2^\nu.\Omega)}^2}{{\theta_2 \carac
  \epsilon_1, {\epsilon_2} (0, 2^\nu p^n .\Omega)}^2}=\pi_1 \ldots \pi_g.$$
\end{lemma}
\begin{proof}
  Let $A'_{\Z_q}$ be the quotient of $A_{\Z_q}$ by $K_1(\pol^{2^\nu
    p}_{\Z_q})[2^\nu]$ the maximal $2^\nu$-torsion subgroup of $K_1$.
  As $K_1(\pol^{2^\nu p}_{\Z_q})[2^\nu]$ is an isotropic subgroup of
  $K(\pol^{2^\nu p}_{\Z_q})$ for the commutator pairing, the line
  bundle $\pol^{2^\nu p}_{\Z_q}$ descends to a line bundle
  $\pol'^{p}_{\Z_q}$ on $A'_{\Z_q}$ which comes with a Lagrangian
  decomposition $K(\pol'^p_{\Z_q}) = K_1(\pol'^p_{\Z_q}) \times
  K_2(\pol'^p_{\Z_q})$ and a theta structure $\Theta'_p$ of type $Z_p$
  inherited from $\Theta_{2^\nu p}$ by~\cite[Prop.2]{MR34:4269}.

  We remark that $A'_{\Z_q}$ being the quotient of $A_{\Z_q}$ by an
  \'etale subgroup is a canonical lift of its special fiber
  $A'_{\F_q}$.  As before, we can consider $A'_\C = A'_{\Q_p}
  \otimes_{\psi} \C$ and we have an isomorphism of analytic varieties
  $j' : A'_\C \rightarrow A'_{an} = \C^g / \Lambda_{ 2^\nu \Omega}$.
  Let $\kappa' : \C^g \rightarrow A'_\C$ be the canonical projection.
  By the choice we have made on $\Omega$, the $p-$torsion points of
  $A'_{an}$ given by $\kappa'(1/p.\Z^g)$ correspond via $j'^{-1}$ to
  a canonical lift of the maximal \'etale quotient of $A'_{\Z_q}[p]$.

  We can then consider the analytic variety ${A'}_{an}^{n} = \C^g /
  \Lambda_{p^n 2^\nu \Omega}$. The inclusion of lattices $\Lambda_{p^n
    2^\nu \Omega} \subset \Lambda_{2^\nu \Omega}$ gives an isogeny
  $\iota: {A'}_{an}^n \rightarrow A'_{an}$. Using exactly the same
  proof as in~\cite[pp.78]{Ritzenthaler03}, one obtains that $\iota$ is a
  lift of the Frobenius morphism acting on $A'_k$, that
  ${A'}_{an}^n$ and $A'_{an}$ are two representatives of the same
  class element of $\Si_g / \Gamma_g(p)$. Moreover, for all $\epsilon_1,
  \epsilon_2 \in Z_2$ we have
$$\theta_2 \carac \epsilon_1, {\epsilon_2} (0, 2^\nu.\Omega)^2= (\pi_1\ldots
\pi_g) \theta_2 \carac
  \epsilon_1, {\epsilon_2} (0, 2^\nu p^n .\Omega)^2,$$
where $\pi_1, \ldots, \pi_g$ are the $g$ Eigenvalues of the $\ell$-adic
representation of the Frobenius morphism acting on $A'_{\F_q}$ which are
units modulo $p$.

The hypothesis that $\Theta_{2^\nu}$ is defined over $\Z_q$ implies
that $K_1(\pol_{\F_q}^{2^\nu p})[2^\nu]$ is defined over $\F_q$. As a
consequence, the two abelian varieties $A_{\F_q}$ and $A'_{\F_q}$ are
$\F_q$-isogeneous and, using a theorem of Tate~\cite{MR34:5829}, we
deduce immediately that they have the same characteristic polynomial
of the Frobenius morphism.
\end{proof}

\begin{lemma}\label{lemma3}
  Let $\gamma : Z_{2^\nu} \rightarrow Z_{2^\nu p}$, $j \mapsto pj$.
  For each $\chi \in \hat{Z}_{2^\nu p}$ character of order $2$, there
  exists $\epsilon \in Z_2$ such that
$$\sum_{u\in Z_{2^\nu}} \chi(\gamma(u)) \theta_{2^\nu p} \carac 0, {\gamma(u)}
(0, 1/(2^\nu p).\Omega) =
\theta_2 \carac \epsilon,0(0, (2^\nu /p). \Omega).$$
We have also:
$$\sum_{u\in Z_{2^\nu p}} \chi(u) \theta_{2^\nu p} \carac 0, {u} (0,
1/(2^\nu p).\Omega) =
\theta_2 \carac \epsilon,0(0, 2^\nu p.\Omega).$$
\end{lemma}
\begin{proof}
For $l \in \N^*$, $a,b \in Z_l$ and $\Omega_0 \in \Si_g$, we put:
\begin{eqnarray*}
f_a & = & \theta_1\carac a/l,0 (lz, l.\Omega_0) \\
g_b & = & \theta_1\carac 0,{b/l} (z,l^{-1}.\Omega_0) \\
\end{eqnarray*}

Then we have the following formula (see~\cite[pp.124]{MR85h:14026}):
\begin{equation}\label{rel}
f_a = \sum_{ a \in Z_l} \exp(-2\pi i \frac{ab}{l}) g_b.
\end{equation}
Let $\chi \in \hat{Z}_{2^\nu p}$ be a character of order $2$ and let
$\epsilon \in Z_2$ be such that for all $u \in Z_{2^\nu p}$, we have
$\chi(u)=\exp(-\pi i \epsilon u)$. The lemma is obtained by applying
formula (\ref{rel}) with $l=2^\nu$, $\Omega_0=p\Omega$ and then with
$l=2^\nu p$ and $\Omega_0=\Omega$.
\end{proof}

We are ready to prove Proposition \ref{formula1}. Let $\gamma' : Z_{2}
\rightarrow Z_{2^\nu p}$, $j \mapsto 2^{\nu -1}pj$.
By applying
successively Lemma \ref{lemma1} and Lemma \ref{lemma3}, we obtain that
for an element $\delta \in Z_2$ and $\chi \in Z_{2^\nu p}$ a character
of order $2$ we have
\begin{eqnarray*}
  \psi\left(\frac{\sum_{u \in Z_{2^\nu}}a_u}{\sum_{u \in Z_{2^\nu p}}a_u}
  \right) & = & \frac{ \sum_{u \in Z_{2^\nu}} \chi(u) \theta_{2^\nu p} \carac 0, {u
      + \gamma'(\delta)} (0, 1/(2^\nu p) .\Omega)}{\sum_{u \in
      Z_{2^\nu p}} \chi(u)
    \theta_{2^\nu p} \carac 0, {u + \gamma'(\delta)}(0, 1/(2^\nu p) .\Omega) } \\
  & = & \frac{ \sum_{u \in Z_{2^\nu}} \chi'(u) \theta_{2^\nu p} \carac 0, {u} (0,
    1/(2^\nu p) .\Omega)}{\sum_{u \in Z_{2^\nu p}} \chi'(u) \theta_{2^\nu p}
    \carac 0, {u}(0, 1/(2^\nu p) .\Omega) } \\
  & = & \frac{ \theta_2 \carac \epsilon, 0 (0, (2^\nu /p). \Omega)}{\theta_2 \carac \epsilon, 0
    (0, (2^\nu p) .\Omega)},
\end{eqnarray*}
where $\chi'(u)=u+\gamma'(\delta)$ and $\epsilon$ is chosen such that
for all $u \in Z_{2^\nu p}$ we have $\chi'(u)=\exp(-\pi i \epsilon u)$.
The second equality is due to the fact that $\Delta \in A_\C \cap K_1(\pol_\C)$.

On the other side we have
\begin{eqnarray*}
  \Norm_{\Q_q / \Q_p}\left(\psi^{-1} \left( \frac{ \theta_2 \carac \epsilon,
        0 (0, (2^\nu /p). \Omega)}{\theta_2 \carac \epsilon, 0
        (0,2^\nu p .\Omega)} \right) \right) & = & \psi^{-1} \left(
    \frac{ \theta_2 \carac \epsilon, 0 (0, (2^\nu /p^n). \Omega)}{\theta_2 \carac \epsilon, 0
      (0, 2^\nu p^n. \Omega)} \right) \\
  & = & \psi^{-1} \left( \frac{ \theta_2 \carac \epsilon, 0 (0, 2^\nu \Omega)}{\theta_2 \carac \epsilon, 0
      (0, 2^\nu  p^n .\Omega)} \right)^2 \\
  & = & \pi_1 \ldots \pi_g,
\end{eqnarray*}
by Lemma \ref{lemma2}.

\section{Description of the algorithm}
\label{algo}

In this section we explain how to use the formulas given in Section
\ref{thetaidentities} in order to count points on the Jacobian of a
generic ordinary hyperelliptic curve over a finite field of odd
characteristic.  Assume that we have chosen a prime $p>2$ and an
integer $g \geq 1$.

\begin{theorem}\label{main}
  Let $C$ be an hyperelliptic curve of genus $g$ with all Weierstrass
  points rational over a finite field $\F_q$ of characteristic $p$
  such that the Jacobian $J(C)$ is ordinary and absolutely simple. Let
  $\nu$ be an integer greater or equal $3$, we suppose that the
  $2^\nu$-torsion of $J(C)$ is defined over $\F_q$. The algorithm for
  the computation of the number of $\F_q$-rational points $\# C(\F_q)$
  of the curve $C$, that we give in the following, has asymptotic time
  complexity $O(n^{2+o(1)})$ and asymptotic space complexity $O(n^2)$
  where $n = \log(\#\F_q)$.
\end{theorem}

From Theorem \ref{main}, we deduce
\begin{corollary}
  Let $C$ be an hyperelliptic curve of genus $g$ over a finite field
  $\F_q$ of characteristic $p$ such that the Jacobian $J(C)$ is
  ordinary and absolutely simple. There exists an algorithm to compute
  the number of $\F_q$-rational points $\# C(\F_q)$ of the curve $C$
  which has asymptotic time complexity $O(n^{2+o(1)})$ and asymptotic
  space complexity $O(n^2)$ where $n = \log(\#\F_q)$.
\end{corollary}
\begin{proof}
  Let $\nu$ be an integer greater or equal $3$. Let $\F_{q^r}$ be an
  extension of $\F_q$ and consider $C_{\F_{q^r}}$ the curve obtained
  from $C$ by doing a base field extension from $\F_q$ to $\F_{q^r}$.
  We suppose that $r$ is chosen such that the $2^{\nu}$-torsion points
  of $J(C_{\F_{q^r}})$ are defined over $\F_{q^r}$. Using a rational
  expression of the group law on $J(C)$, we see that there exists a
  bound on $r$ which is independent of the choice of $C$ when $g$ is
  fixed.

  Applying Theorem \ref{main} we obtain in time $O(n^{2+o(1)})$ the
  characteristic polynomial $\chi_{F'}$ of the $q^r$-Frobenius
  morphism $F'$. Let $\alpha'_1, \ldots , \alpha'_{2g}$ be the roots
  of $\chi_{F'}$. On the other side, let $\alpha_1, \ldots ,
  \alpha_{2g}$ be the roots of $\chi_F$ the characteristic polynomial
  of the $q$-Frobenius acting on $C$. We have by~\cite{MR1251961}
  Theorem V.1.15, ${\alpha'}_i^r=\alpha_i$.  By computing the roots
  $\alpha'_1, \ldots , \alpha'_{2g}$ and taking their $r^{th}$ root,
  we obtain a finite set of possible roots for $\chi_F$ up to
  permutation of the indices. In order to finish the proof, we just
  have to remark that all the above computations for a fixed genus
  have constant complexity with respect to $\log(q)$. Moreover, it is
  possible to check the result of the computations in quasi-quadratic
  time by taking a point $P$ of $J(C)$ and computing $\lambda.P$ where
  $\lambda$ is the supposed group order of $J(C)$.
\end{proof}

We remark that the existence of such a quasi-quadratic time algorithm
in the special case $p=2$ is proved in~\cite{LL07}.  In the following
we give an algorithm which is expected to have the desired properties.
In the case that we take $\nu=1$ in the statement of Theorem
\ref{main}, we have verified that the correctness of the algorithm
still holds by counting points on many examples of elliptic curves in
characteristic $3$ and $5$ and on some genus $2$ curves in
characteristic $3$.  Our algorithm follows the so-called lift and norm
paradigm which was introduced by Satoh in ~\cite{MR2001j:11049}. The
algorithm is as follows.

We assume that the hyperelliptic curve $C$ is given by an equation
of the form
\[
y^2 = \prod_{i=1}^{2g+2}(x-\overline{\alpha}_i)
\]
where $\overline{\alpha_i} \in \F_q$.

\paragraph{Initialization phase:}
Let $J(C)$ be the Jacobian of $C$. The aim of this first phase is to
compute the theta null point associated to a semi-canonical product
theta structure $\Theta_{2^\nu p}=\Theta_{2^\nu} \times \Theta_p$ for
$\pol^{2^\nu p}$ (compare Section \ref{sigmasquare})
where $\pol$ is a degree $1$ symmetric ample line
bundle on $J(C)$.

This can be done in the following way. First compute the theta null
point associated to a theta structure $\Theta_2$ of type $Z_2$ for
$\pol^2$. By considering any lift $\mathcal{C}$ of $C$ over $W(\F_q)$
defined by lifts $\alpha_i$ of $\overline{\alpha_i}$ over $\Z_q$
and a given embedding $\psi: \Z_q
\rightarrow \C$ one can view the Jacobian $J(\mathcal{C})$ of the
lifted curve $\mathcal{C}$ as a
complex abelian variety. One can consider a symplectic basis of
$H_1(\mathcal{C}, \Z)$ given by $A$-cycles and $B$-cycles as described
in~\cite{MR86b:14017}.  The associated period matrix $\Omega$ of
$J(\mathcal{C})$ is an element of $\Si_g$, the $g$-dimensional Siegel
upper half plane. For $\epsilon_1, \epsilon_2 \in \N$ and $l \in \N^*$, we
denote by $\theta_l \carac \epsilon_1,{\epsilon_2} (z,\Omega)$ the Riemann theta
function with rational characteristic given by (\ref{thetafunction}).

According to~\cite[pp.124]{MR85h:14026} a theta null point associated
to a well chosen theta structure of the second power of the degree $1$
canonical line bundle defined by $\Omega$ is given by $(a_u)_{u\in
  Z_2}$ with
$$a_u = \lambda \theta_2 \carac 0,{u} (0, 1/2 \Omega),$$ 
where $\lambda \in \C^*$. This theta null point, which correspond to
the case $\nu=1$, can be computed in two steps.

\paragraph{Step 1.}  
For $i=1\ldots g$, let $\tau_i$ be the vector $(\tau_{i,j})_{j \in
  \{1, \ldots , g\}}$ such that $\tau_{i,j}=0$ if $j<i$ and
$\tau_{i,j}=1$ if $i \geq i$. Using the Thomae-Fay
formulas~\cite[pp.121]{MR86b:14017}, we compute
$$
\theta_1 \carac v,u (0,\Omega)^2 = \pm \sqrt{\prod_{0\leqslant i < j \leqslant g
  } (\alpha_{2i+e_i}-\alpha_{2j+e_j})
  (\alpha_{2i+1-e_i}-\alpha_{2j+1-e_j})},
$$
where $e_0 = 0$ and the vector $(e_i)_{i=1\ldots g} \in \F_2^g$ is
given by $(e_i)=u+ \sum_{i=1}^g v_i.\tau_i$, for $i=1, \ldots , g$,
where $v=(v_i) \in \F_2^g$. Here we choose the sign of the square root
at random.

\paragraph{Step 2.} Case $\nu =1$. We proceed to a reverse duplication
step which can be done according to the Riemann duplication formulas
\cite{MR49:569} by finding $(a_u)_{u\in Z_2}$ such that
$$
\theta_2 \carac 0, {u} (0,\Omega)^2 = \frac{1}{2^g} \sum_{v \in
  Z_2} a_{v+u} a_v.
$$
This algebraic system may be solved by using the Groebner basis
algorithm and by picking up any solution. We check that we obtain a
valid theta null point by computing the associated $4$-theta null
point and verify that it satisfies the level-$4$ Riemann type
equations (compare with Section \ref{riemanntwop}). If this is not the
case, we go back to Step one and choose different signs for the square
roots.

Let $\Sp(2g, \Z)$ be the group of symplectic matrices acting on
$\mathbb{H}_g$.  Denote by $\Gamma_2$ the subgroup of $\Sp(2g, \Z)$
consisting of the elements $\gamma \in \Sp(2g, \Z)$ such that $\gamma
\equiv I_{2g} \bmod 2$ where $I_{2g}$ is the identity matrix of
dimension $2g$. The resulting theta null point $(a_u)_{u \in Z_2}$ has
the property that if we raise to the fourth power the coordinates of
its image by the Riemann duplication formula, we recover the values
deduced from the ramification points $\alpha_i$ of $\mathcal{C}$ by
the Thomae formulas. According to~\cite[pp.3.131]{MR86b:14017} this
means that $(a_u)_{u \in Z_2}$ is the theta null point associated to
the second power of a degree one symmetric ample line bundle defined
by $\Omega'$ where $\Omega' = \gamma.  \Omega$ for an element $\gamma
\in \Gamma_2$.

As all the computations described in this paragraph are algebraic,
they can be made directly in $\Z_q$ using the embedding $\psi$, and
even in $\F_q$ as $\mathcal{C}$ has good reduction modulo $p$. This
procedure gives the computation of a theta null point $(a_u)_{u\in
  Z_2}$ for a symmetric theta structure $\Theta_2$ associated to the
second power of a degree $1$ ample symmetric line bundle $\pol$ on
$J(C)$.  It should be noted that we have to assume that $\Theta_2$ is
rational over $\F_q$ in order to have that $a_u \in \F_q$, for $u\in
Z_2$.

Now, we describe a variation of Step 2 to cover the case $\nu >1$.
\paragraph{Step 2'.} Case $\nu > 1$.  From the knowledge of
$\theta_{1} \carac v,u(0, \Omega)$, we proceed to two reverse
duplication steps which can be done by finding successively for
$i=1,2$, $u,v \in Z_2$,
$\theta_{2} \carac v,u(0, (1/2^i)\Omega)$ such that
$$\theta_{2} \carac v,u(0,(1/2^{i-1})\Omega)^2=\frac{1}{2^g} \sum_{t\in
  Z_2} (-1)^{^tv t} \theta_{2} \carac 0,{u+t}(0, (1/2^{i})\Omega)
\theta_{2} \carac 0,u (0, (1/2^{i})\Omega).$$ 
This algebraic system can easily be solved by using the Groebner basis
algorithm and by picking up any suitable solution, we obtain $\theta_2
\carac u,v(0, (1/4).\Omega)$. If $v \in Z_2$, denote by $\hat{v}$ the
element of $\hat{Z}_2$ defined by $\hat{v} : Z_2 \rightarrow \{-1 , 1
\} \subset \Z$, $z=(z_i) \mapsto (-1)^{\sum_{i=1}^g z_i v_i}$.

On the other side, we have $Z_2 \simeq Z_4 / Z_2$ and let $\phi: Z_2
\simeq Z_4 / Z_2 \rightarrow Z_4$ be a section of the canonical
projection. Let $(b_u)_{u \in Z_4}$ be defined such that
$$\theta_2 \carac v,u (0, (1/4).\Omega)=\sum_{t \in Z_2} \hat{v}(t)
b_{\phi(u)+t},$$
where $Z_2$ is considered as a subgroup of $Z_4$ via $j \mapsto 2j$.
We can compute $(b_u)_{u \in Z_4}$ from the knowledge of $\theta_4 \carac
u,v(0, \Omega)$ by solving a linear system of fixed size.

We know~\cite[pp.334]{MR34:4269}, that $(b_u)_{u \in Z_4}$ is the
theta null point of $J(C)$ associated to a symmetric theta structure
of type $Z_4$. Now, plugging $(b_u)_{u \in Z_4}$ into the relations
given of the Riemann equations of level $2^\nu$ (compare with Section
\ref{riemanntwop}) together with the symmetric relations, we know
by~\cite[pp.87]{MR36:2621} that the so obtained system admits a unique
solution $(a_u)_{u \in Z_{2^\nu}}$ which may easily be computed using
a Groebner basis algorithm.

\paragraph{Step 3.} In the following we explain how to compute a level
$2^\nu p$-theta null
point from the above $2^\nu$-theta null point. We use the notation of Section
\ref{extheta2p}. Let $I$ be the ideal of the multivariate
polynomial ring $\F_q[x_u|u \in Z_{2^\nu p}]$ which is spanned by the
relations of Theorem
\ref{riemannrelations} together with the symmetry relations
$a_u = a_{-u}$ for $u \in Z_{2^\nu p}$.
We find $v \in Z_2$ such that $a_v$ is a unit.
Let $J$ be the image of $I$ under the evaluation map
\begin{eqnarray*}
\F_q[x_u|u \in
Z_{2^\nu p}] \rightarrow \F_q[x_u|u \in
Z_{2^\nu p},2^\nu u \not=0], \quad x_u \mapsto \left\{
\begin{array}{l@{,\hsp}l}
\frac{a_u}{a_v} & u \in Z_{2^\nu} \\
\frac{x_u}{a_v} & \mathrm{else}
\end{array} \right.
\end{eqnarray*}
If we chose an order on the set of the remaining variables $x_u$,
$u\in Z_{2^\nu p} \setminus Z_{2^\nu}$, it defines a well-ordered
lexicographic monomial basis on $J$. One can compute a reduced
Groebner basis for $J$ with respect to this monomial order. By
Theorem \ref{multiplicity}, the closed subscheme of
$\mspec(\F_q[x_u|u \in Z_{2^\nu p},2^\nu u \not= 0])$ defined by $J$
is of dimension $0$. The last polynomial of this reduced Groebner
basis is a univariate polynomial $f(x)\in \F_q[x]$ and by
\cite{BeMoMaTr94}, we generically have
$$\F_q[x_u | u \in Z_{2^\nu p}, 2^\nu u \neq 0 ] / J \simeq \F_q[x] /(f),$$
where the degree of $f$ is uniformly bounded by a function of $g$ and
$p$ which is constant with respect to the complexity parameter
$\log_p(q)$. According to Proposition
\ref{multiplicity}, one can pick up a solution $(a_u)_{u \in
  Z_{2^\nu p}}$ corresponding to the root of $f$ with multiplicity $p^g$.

\paragraph{Lift phase}
Let $(a_u)_{u \in Z_{2^\nu p}}$ with $a_u \in \F_q$ the null point obtained
from the initialization phase.
Let $\mathcal{R}$ be the set of polynomials in $\Z_q[x_u, y_u | u \in
Z_{2^\nu p}]$ deduced
from the relations of Theorem \ref{thetarelations} and Theorem
\ref{riemannrelations}, where in the Riemann type relations $a_u$ is
replaced by $y_u$ for all $u \in Z_{2^\nu p}$, and in the Frobenius
type relations, $a_u$ and $a_u^{\sigma^2}$ are
replaced by $x_u$ and $y_u$, respectively, for all $u\in Z_{2^\nu p}$. We put $x_0 = y_0 =1$ and use
the symmetry relations in order to
obtain a set of multivariate polynomials depending on $1/2 [
(2^\nu p)^g-2^{\nu g}]+2^{\nu g}-1$ variables $x_u$ and a subset of the same
cardinality of the coordinates
$y_u$.

Pick up any subset $1/2 [(2^\nu p)^g-2^{\nu g}]+2^{\nu g}-1/2
[g(g+1)]-1$ of Riemann type equations and $1/2 [g(g+1)]$ Frobenius
type equations to form an application $$\Phi: \Z_q^{2^{\nu
    g-1}(p^g+1)-1} \times \Z_q^{2^{\nu g-1}(p^g+1)-1} \mapsto
\Z_q^{2^{\nu g-1}(p^g+1)-1}.$$ For a suitable choice of the Riemann
and Frobenius equations, the conditions of~\cite[Th.2]{LL07} are
satisfied and one can use the lifting algorithm given ibid in order to
lift in a canonical way the theta null point $(a_u)_{u \in Z_{2^\nu
    p}}$ to obtain the canonical theta null point $(b_u)_{u \in
  Z_{2^\nu p}}$ of the canonical lift with $b_u \in \Z_q$.

\paragraph{Norm phase}
We use the notation of Section \ref{gtrace}.
By Proposition \ref{formula1}, one computes the
product of the Eigenvalues $\pi_1, \ldots , \pi_2$ of the absolute $q$-Frobenius
morphism $F$, which are units modulo $p$, as
$$
  \pi_1 \ldots \pi_g = \Norm_{\Q_q / \Q_p} \left( \frac{\sum_{u \in
        Z_2}b_u}{\sum_{u \in Z_{2^\nu p}}b_u} \right).$$

\paragraph{Reconstruction phase}
The problem here is to be able to recover $\chi_{F}$ where
from the
knowledge of $\lambda = \pi_1 \ldots \pi_g$ computed up to a certain
precision $m$. If the genus $g$ of $C$ is one, then $\chi_F$ is
immediately computed from $\pi_1$. In the case that the curve $C$ has genus $2$ one
can use the formulas described in~\cite{Ritzenthaler03}.

From now on, we suppose that $g \geq 2$.
Following~\cite{Ritzenthaler03, LL07}, one can use the LLL algorithm
in order to recover the symmetric polynomial of $C$ considered as a
curve over $\F_q$ that we denote by $P_{sym}$. By definition, the
symmetric polynomial of $C$ is the unitary degree $2^{g-1}$ polynomial
whose roots are $x + q^g /x$ where $x$ runs over all products of $g$
terms taken successively in the pairs $\{ \pi_1, q/ \pi_1 \}, \ldots ,
\{ \pi_g, q/ \pi_g \}$. It is easy to see that $P_{sym}$ is a
polynomial with coefficients in $\Z$ and that there exists a quick
algorithm, at least when $\chi_{F}$ is irreducible, to compute
$\chi_{F}(\pm X)$ from the knowledge of $P_{sym}$ (see
\cite{Ritzenthaler03}). By~\cite{MR34:5829}, $\chi_F$ is irreducible
when the Jacobian of $C$ is absolutely simple, and this last condition is
generic. A last check on the curve allows us to obtain $\chi_{F}$.
We explicitly determine bounds on the precision $m$ needed when the
genus increases.

The computation of $P_{sym}$ from $\eta = \lambda + q^g /\lambda$, can
be done by LLL reducing the lattice whose basis vectors are given by
the columns of the following matrix:
\begin{displaymath}
  \left[
    \begin{array}{ccccc}
      \Upsilon \times M_0 & \Upsilon \times M_1 & \cdots & \Upsilon \times
      M_{2^{g-1}+1}  & \Upsilon \times p^m\\
      0 & 0    & \cdots &  p^{\lfloor n\times S_{2^{g-1}+1}
      \rfloor} & 0\\
      0 & 0  & \cdots & 0 & 0\\
      \vdots & \vdots & \ddots & \vdots & \vdots\\
      0 & p^{\lfloor n\times S_2 \rfloor} & \cdots & 0 & 0\\
      p^{\lfloor n\times S_0 \rfloor} & 0    & \cdots & 0  & 0
    \end{array}
  \right],
\end{displaymath}
where
\begin{eqnarray*}
[M_i]_{i=0, \ldots, 2^{g-1}} = \left[\ p^{(2^{g-1}-1-i)n}\eta^i \bmod 2^m\ |\ i \in \{0,\ldots, 2^{g-1}-1\}\right] \\
\cup [\eta^{2^{g-1}}\bmod 2^m].
\end{eqnarray*}
and
\begin{displaymath}
  [S_i]_{i=0, \ldots, 2^{g-1}+1} = \left[\frac{(i-1)(g-2)}{2}\ |\ i \in \{0,\ldots,
    2^{g-1}-1 \}\right]\cup \left[\frac{2^{g-1}(g-2)}{2}+1\right]
\end{displaymath}
where $\Upsilon$ is some arbitrarily large constant.
The power of $p$ appearing in the $M_i$ are meant to take into account
the valuation of the coefficients of $P_{sym}$ while the $S_i$ offset
the difference between the modulus of the coefficients of $P_{sym}$.
This matrix can be used as long as $2^g<n$.

The coefficients of $P_{sym}$ are components of a vector $\Pi$ of
small norm in ${\mathcal L}$. Asymptotic estimates state that a
lattice reduction using the LLL
algorithm~\cite{LenstraLeLo82,MR99b:94027} can compute it if its
euclidean norm $||\ ||_2$ (or sup-norm $||\ ||_1$) satisfy
\begin{math}
  ||\Pi||_1 \leqslant ||\Pi||_2 \leqslant \det({\mathcal L})^{1/\dim {\mathcal L}}.
\end{math}
Since we can evaluate, on the first hand, the norm $||\ ||_1$ of $\Pi$
as a function of $n$ and $g$ using the Riemann hypothesis for curves
and, on the other hand, the determinant of ${\mathcal L}$ as a
function of $m$, $g$ and the size of $\Upsilon$ (product of the
elements on diagonal), this yields
\begin{displaymath}
  m > n\, \left[ \ln(p) \, g^2 \, 2^{3g-5} - \frac{\ln(p)}{\ln(2)} \,
    (g-2) \, 2^{2(g-1)} \right].
\end{displaymath}

From the knowledge of the roots of $P_{sym}$, it is possible to
recover the set $\{ \pi_i^2, i=1\ldots g
\}$~\cite[pp.119]{Ritzenthaler03} where the $\pi_i$ are the roots of
$\chi_{F}$ which are units modulo $p$. In the case that $\chi_{F}$ is
irreducible, we immediately deduce $\chi_{F}$ from the knowledge of
its roots. In order to remove the sign ambiguity it remains to
determine whether the order of the Jacobian is $\chi_F(1)$ or
$\chi_F(-1)$ by multiplying points with possible group orders.

\section{Complexity analysis}
\label{complexity}

In this section, we give a complexity analysis of the previously
described algorithm.

\paragraph{Initialisation phase} The dominant complexity for this
phase is the Groebner basis computation of Step $3$. Let $J$ be as in
Section \ref{extheta2p}.

Let $D$ be the degree of the ideal $J$. According to~\cite{Laz3}, the
computation of a Groebner basis with respect to a lexicographic
monomial order can be done by doing a Gaussian elimination on a matrix
of dimension given by the number of monomials of degree $D$.  The
theorem of Bezout gives a bound on $D$ which is the product of the
degrees of the polynomials generating the ideal $J$. As the number of
polynomial relations defined by Theorem \ref{riemannrelations} depends
only on $g$ and $p$ and the degree of these relations is constant, $D$
is fixed as long as $g$ and $p$ are. This means that the Groebner
basis can be computed by doing a Gaussian elimination on a matrix of
fixed dimension whose coefficients are in $\F_{q}$. This requires
$O(n^\mu)$ time operations where $n$ and $\mu$ have been defined at
the end of the introduction.

We remark that~\cite{Laz2} gives a much finner bound on the degree of
the Groebner basis if one chooses for the monomial order of $J$ a graded
reverse lexicographic order. As a consequence, one should better first
compute a Groebner basis of $J$ for a graded reverse lexicographic
order and then use the FGLM algorithm in order to perform the change
of order towards a lexicographic order.

\paragraph{Lift phase}
In the case that the base field admits a Gaussian Normal Basis one can
lift in time $O(\log(n) m^\mu n^\mu )$ using the algorithm~\cite{LL03}.
In the general case, one can use the algorithm of Harley which is in
$O(\log(m) m^\mu n^\mu)$ time complexity~\cite[pp.254]{MR2162716}.

\paragraph{Norm phase}
In the case that the base field admits a Gaussian Normal Basis of type
$t$, H. Y. Kim et al. described an algorithm of the type ``divide and
conquer'' in order to compute such a norm. This algorithm has time
complexity $O(\log (n)m^{\mu} n^{\mu})$.  For the general case, one
can use the algorithm described in \cite[pp.263]{MR2162716} in order
to compute the norm in time $O(\log(n) m^\mu n^\mu)$.

\paragraph{Reconstruction phase}
For fixed genus, the LLL step consists in applying LLL to a lattice of
fixed dimension. Its complexity is the size of the coefficients of the
matrix times the cost for one integer multiplication. This yields,
with asymptotically fast algorithms for multiplying integers, a
$O(m^{1+\mu})$ complexity in time. The cost of the second step is
determined by the computation of roots of polynomials over $\C_p$ and
requires $O(m^{\mu})$. Finally, checking that the order of the
Jacobian is $\chi_F(\pm 1)$ needs $O(m)$ applications of the group
law, that is to say a complexity in time equal to $O(mn^\mu)$ with
Cantor formulas~\cite{MR88f:11118}.

\section{A finiteness theorem}\label{secfinit}
This section is devoted to the proof of Proposition
\ref{multiplicity}. In Section \ref{riemanntwop}, we recall several
equivalent presentations of the Riemann equations which are used in
the course of the proof given in Section \ref{theproof}.

We first fix some notations. Let $A$ be an abelian variety over a
field $k$ and $\pol$ be an ample symmetric line bundle over $A$. Let
$\Theta_\ell$ be a theta structure for $\pol$ of type $Z_\ell$.
Let $(\theta^{\Theta_\ell}_i)_{i \in Z_\ell}$ be a basis of the
global sections of $\pol$ determined by the theta structure
$\Theta_\ell$ and let $x$ be a closed point of $A$. Denote by
$\mathscr{O}_{A}$ the structure sheaf of $A$ and let $\rho:
\mathscr{O}_{A,x} \rightarrow k'$ be the evaluation morphism onto the
residual field $k'$ of $x$.  We can choose an isomorphism $\xi_x:
\pol_x \simeq \mathscr{O}_{A,x}$. For all $i\in Z_\ell$ the
evaluation of the section $\theta^{\Theta_\ell}_i$ in $x$ is
\begin{equation}\label{eval}
\theta^{\Theta_\ell}_i(x,\xi_x)=\rho \circ
\xi_x(\theta^{\Theta_\ell}_i).
\end{equation}
The resulting projective point that we denote by
$(\theta^{\Theta_\ell}_i(x))_{i \in Z_\ell}$ over $\overline{k}$
does not depend on the choice of the isomorphism $\xi_x$.

\subsection{Riemann's equations revisited}
\label{riemanntwop1}
Let $A$ be a $g$ dimensional ordinary abelian variety over a finite field
$\F_q$ of characteristic $p>2$. Let $\pol$ be an ample symmetric degree $1$ line
bundle on $A$.  Let $\nu>0$ be an integer and $\ell$ be an odd prime
number which can be equal to $p$. Assume that we are given a
symmetric theta structure $\Theta_{2^\nu \ell}$ of type $Z_{2^\nu \ell}$ for
the line bundle $\pol^{2^\nu \ell}$.  The data of $\Theta_{2^\nu \ell}$
defines a basis of global sections of $\pol^{2^\nu \ell}$ that we denote
by $(\theta_u)_{u \in Z_{2^\nu \ell}}$ and as a consequence, a
projective embedding of $A$ in $\proj^{(2^\nu \ell)^g-1}$.

We denote the theta null point with respect to the theta structure
$\Theta_{2^\nu \ell}$ by $(a_u)_{u \in Z_{2^\nu \ell}}$.
The Riemann's equations for level $2^\nu \ell$ are given by the following
theorem

\begin{theorem}
\label{riemannquad}
For all $x,y,u,v \in Z_{2^{\nu+1} \ell}$ which are congruent modulo
$Z_{2^\nu \ell}$, and all $l \in \hat{Z}_2$, we have
\begin{multline}\label{pone}
\big(\sum_{t \in Z_2} l(t) \theta_{x+y+t}*\theta_{x-y+t}\big).\big(\sum_{t \in Z_2} l(t)
a_{u+v+t} a_{u-v+t}\big)= \\
=\big(\sum_{t \in Z_2} l(t) \theta_{x+u+t}*\theta_{x-u+t}\big).\big(\sum_{t \in Z_2} l(t)
a_{y+v+t} a_{y-v+t}\big). 
\end{multline}
\end{theorem}
\begin{proof}
By \cite{MR34:4269}[pp. 339], for all $x,y \in Z_{2^{\nu+1} \ell}$ such that
$x+y \in Z_{2^\nu \ell}$ and for all $l \in \hat{Z}_2$, we have
\begin{equation}\label{ptwo}
\sum_{t \in Z_2} l(t) \theta_{x+y+t}*\theta_{x-y+t}
=\big(\sum_{t \in Z_2} l(t).a_{y+t} \big). \big(
  \sum_{t \in Z_2} l(t). \theta_{x+t} \big).
\end{equation}
In particular, we have, 
\begin{equation}\label{pthree}
\sum_{t \in Z_2} l(t) a_{x+y+t}.a_{x-y+t} =\big(\sum_{t \in Z_2} l(t).a_{y+t} \big). \big(
  \sum_{t \in Z_2} l(t). a_{x+t} \big).
\end{equation}
Now, using (\ref{ptwo}) and (\ref{pthree}) the left hand side of
(\ref{pone}) can be written as
\begin{equation}\label{peq1}
\big[ \big(\sum_{t \in Z_2} l(t).a_{y+t} \big). \big(
\sum_{t \in Z_2} l(t). \theta_{x+t} \big) \big] \big[ \big(\sum_{t \in
  Z_2} l(t).a_{u+t} \big). \big( \sum_{t \in Z_2} l(t). a_{v+t} \big)
\big].
\end{equation}
In the same manner the right hand side of (\ref{pone}) can be written as
\begin{equation}\label{peq2}
\big[ \big(\sum_{t \in Z_2} l(t).a_{u+t} \big). \big(
\sum_{t \in Z_2} l(t). \theta_{x+t} \big) \big] \big[ \big(\sum_{t \in
  Z_2} l(t).a_{y+t} \big). \big( \sum_{t \in Z_2} l(t). a_{v+t} \big)
\big].
\end{equation}
Obviously, (\ref{eq1}) and (\ref{eq2}) are equal.
\end{proof}

In the following, we suppose that $k$ is the field of definition of
$(a_u)_{u \in Z_{2^\nu \ell}}$.  In this section, we denote by
$I_{\Theta_{2^\nu \ell}}$ the ideal of $k[x_u | u \in Z_{2^\nu \ell}]$
generated by the relations of Theorem \ref{riemannquad} where the
$\theta_u$ are replaced by $x_u$. It is proved in
\cite[$\S$4]{MR34:4269} that if $\nu \geq 2$ and $\ell \neq p$, $A$ is
isomorphic to the closed projective sub-variety of $\proj_k^{(2^\nu
  \ell)^g-1}$ defined by the homogeneous ideal $I_{\Theta_{2^\nu
    \ell}}$.

We recover Theorem \ref{riemannrelations} from Theorem
\ref{riemannquad}, by evaluating at the point $O$ of $A$ the sections
of $\pol^{2^\nu \ell}$.  The relations of Theorem \ref{riemannrelations} can be
reformulated, by considering the matrix
$$M=\left( \begin{matrix}
1 & 1 & 1 & 1 \\
1 & 1 & -1 & -1 \\
1 & -1 & 1 & -1 \\
1 & -1 & -1 & 1 \\
\end{matrix}
\right),
$$
and $(x_1,y_1,u_1,v_1),(x_2,y_2,u_2,v_2) \in (Z_{2^\nu \ell})^4$ such that
$$2 \left( \begin{matrix} x_2 \\ y_2 \\ u_2 \\ v_2 \end{matrix} \right) = M 
\left( \begin{matrix} x_1 \\ y_1 \\ u_1 \\ v_1 \end{matrix} \right).$$
If we suppose moreover that $x_2 + y_2 \in 2Z_{2^\nu \ell}$ and $u_2+
v_2 \in 2Z_{2^\nu \ell}$, we have
\begin{eqnarray}\label{eqri1}
\big(\sum_{t \in Z_2} l(t) a_{x_1+t} a_{y_1+t}\big).\big(\sum_{t \in Z_2} l(t)
a_{u_1+t} a_{v_1+t}\big)= \\
=\big(\sum_{t \in Z_2} l(t) a_{x_2+t} a_{y_2+t}\big).\big(\sum_{t \in Z_2} l(t)
a_{u_2+t} a_{v_2+t}\big), 
\end{eqnarray}
$l \in \hat{Z}_2$.

By developing and summing up over all the characters of $\hat{Z}_2$
the Equation (\ref{eqri1}), we obtain
\begin{eqnarray}\label{eqri}
\sum_{t \in Z_2} a_{x+t} a_{y+t} a_{u+t} a_{v+t} = \sum_{t \in Z_2}
a_{x- \tau + t} a_{y + \tau +t} a_{u +\tau +t} a_{v + \tau +t},
\end{eqnarray}
for all $x,y,u, v \in Z_{2^{\nu} \ell}$ and $\tau \in Z_{2^\nu \ell}$ such
that $2 \tau = x-y-u-v$.

These relations can also be presented in their classical form. For
this, we keep the notations of the previous paragraph and suppose from
here that $\nu \geq 2$. Recall that $(a_u)_{u \in Z_{2^\nu
    \ell}}$ denote the theta null point defined by the theta structure
$\Theta_{2^{\nu} \ell}$.  Let $H_{2^{\nu} \ell}=Z_{{2^\nu} \ell} \times
\hat{Z}_{2^{\nu-1}}$ and for all $x=(x',x'') \in H_{{2^\nu} \ell}$, let $b_x= \sum_{t
  \in Z_{2^{\nu-1}}} x''(t) a_{x'+t}$.  Let $H_2=\frac{1}{2} Z_2 \times
(\hat{Z_{2^\nu \ell}})_2= \{ x\in H_{{2^\nu}\ell} | x\,\text{is}\, 2- \text{torsion
  modulo} \, Z_2 \times \{ 0 \} \}$.

\begin{theorem}\label{classical}
  Let $(x,y,u,v) \in H_{{2^\nu} \ell}$ and $\tau=(\tau', \tau'') \in H_{{2^\nu}
    \ell}$ such that $2 \tau = x-y-u-v$ then
\begin{eqnarray*}
  b_{x}b_{y} b_{u}b_{v} =
  \frac{1}{2^g} \sum_{t \in H_2} A(2t')
  b_{x-\tau +t} b_{y+\tau +t} b_{u+\tau+t} b_{v+\tau+t},
\end{eqnarray*}
where $t=(t',t'')$ and $A=\tau''+t''$.
\end{theorem}
\begin{proof}
  It is possible to deduce these relations from (\ref{eqri}) following
  exactly the same computations as~\cite[pp. 334]{MR34:4269}.
\end{proof}

Let $(\theta_u)_{u \in Z_{{2^\nu}}}$ denote the basis of global
sections of $\pol^{{2^\nu}}$ defined by the theta structure
$\Theta_{{2^\nu}}$.  Let $H_{2^\nu} = Z_{2^\nu} \times
\hat{Z}_{2^{\nu-1}}$ and $H_2 = \frac{1}{2} Z_2 \times
(\hat{Z_{2^\nu}})_2$.  For all $x=(x',x'') \in H_{2^\nu}$, set
$\vartheta_x= \sum_{t \in Z_{2^{\nu-1}}} x''(t) \theta_{x'+t}$.  We have the
\begin{theorem}\label{classicalbis}
  Let $P,Q$ be two closed points of $A$.  Denote by $\mathscr{O}_{A}$
  the structure sheaf of $A$.  For $X \in \{ P, Q, P+Q,P-Q, 0 \}$, we
  choose isomorphisms $\xi_X : \pol^{{2^\nu}}_X \simeq
  \mathscr{O}_{A,X}$.  Let $(x,y,u,v) \in H_{2^\nu}$ and $\tau=(\tau',
  \tau'') \in H_{2^\nu}$ such that $2 \tau = x-y-u-v$, we have
\begin{eqnarray*}
  \vartheta_{x}(P+Q,\xi_{P+Q})\vartheta_{y}(P-Q,\xi_{P-Q}) \vartheta_{u}(0,\xi_0)
  \vartheta_{v}(0,\xi_0) =\\
  =  \lambda \frac{1}{2^g} \sum_{t \in H_2} A(2t')
  \vartheta_{x-\tau+t}(P,\xi_P)\vartheta_{y+\tau+t}(P,\xi_P) \vartheta_{u+\tau+t}(Q,\xi_Q)
  \vartheta_{v+\tau+t}(Q,\xi_Q),
\end{eqnarray*}
where $t=(t',t'')$, $A=\tau''+t''$ and $\lambda \in \overline{k}$ is independent of
the choice of $(x,y,u,v) \in H_{2^\nu}$.
\end{theorem}

If $k=\C$, this last formula is exactly~\cite[pp.141]{MR48:3972}.

\subsection{Proof of Theorem \ref{multiplicity}}\label{theproof}
In this section, we denote by $\F_q$ a finite field of
characteristic $p>2$ with $q$ elements. Let $A$ be an ordinary
abelian variety over $\F_q$ and let $\pol$ be an ample symmetric line
bundle of degree $1$ on $A$. Let $\ell$ be an odd prime number
and suppose that we are given a theta structure $\Theta_{2^\nu}$ for
$\pol^{{2^\nu}}$. We denote the theta null point with respect to the theta
structure $\Theta_{{2^\nu}}$ by $(a_u)_{u \in Z_{{2^\nu}}}$. We suppose that 
$(a_u)_{u \in Z_{{2^\nu}}}$ is defined over $\F_q$.

In the following $Z_{2^\nu}$ is considered as a subgroup of $Z_{{2^\nu}\ell}$ via
the map $j \mapsto \ell j$.  Let $I$ be the ideal of the multivariate
polynomial ring $\F_q[x_u|u \in Z_{{2^\nu}\ell}]$ which is spanned by the
relations of Theorem \ref{riemannrelations}, taken modulo $p$,
together with the symmetry relations $x_u = x_{-u}$ for all $u \in
Z_{{2^\nu}\ell}$.  Let $J$ be the image of $I$ under the specialization map
\begin{eqnarray*}
\F_q[x_u|u \in
Z_{{2^\nu}\ell}] \rightarrow \F_q[x_u|u \in
Z_{{2^\nu}\ell},{2^\nu}u \not=0], \quad x_u \mapsto \left\{
\begin{array}{l@{,\hsp}l}
a_u & u \in Z_{2^\nu} \\
x_u & \mathrm{else}
\end{array} \right. .
\end{eqnarray*}

We want to prove that if $\nu \geq 2$, the ideal $J$ defines a
$0$-dimensional affine algebraic set.

\begin{remark}
  In the case that $\nu \geq 3$ and $\ell$ is prime to $p$, the
  preceding theorem can be proved using the general description of the
  moduli space of abelian varieties with a theta marking given in
 ~\cite{MR36:2621}. But this general description is not available
  under the hypothesis that we consider.  It should also be remarked
  that the variety defined by $J$ when $\ell$ is equal to the
  characteristic of $\F_q$ is singular so that it is not possible to
  lift to the $p$-adics to recover the situation where $\ell$ is
  different from $p$. In the following we present a proof which is
  valid both in the situation where $\ell$ is equal to or different
  from the characteristic of $\F_q$.
\end{remark}

Denote by $J'$ the ideal of $\F_q[x_u | u \in Z_{{2^\nu}\ell}]$ generated by
$J$ and elements $x_u -a_u$ for all $u \in Z_{2^\nu}$.  Denote by
$V_{J'}$ the closed sub-variety of the affine space of dimension
$(2^\nu \ell)^g$, $\xa^{({2^\nu}\ell)^g}$ defined by
$J'$. We want to show that $V_{J'}$ is a $0$-dimensional variety.

We recall that the data of $\Theta_{2^\nu}$ gives a basis $(\theta_u)_{u \in
  Z_{2^\nu}}$ of the global sections of $\pol^{2^\nu}$ on $A$ and as a
consequence an embedding of $A$ in $\proj_{\F_q}^{{2^\nu}^g-1}$.  If
we denote by $V_{I_{\Theta_{2^\nu}}}$ the projective variety defined by the
homogeneous ideal $I_{\Theta_{2^\nu}}$, $A$ is isomorphic to
$V_{I_{\Theta_{2^\nu}}}$ as an abelian variety~\cite[$\S$4]{MR34:4269}.

The idea of the proof of the Theorem \ref{multiplicity} is to
interpret the solutions of $J'$ as closed points in the variety
$A=V_{I_{\Theta_{2^\nu}}}$ and then to show that these points are
$\ell$-torsion points of $A$.  This is exactly the content of Lemma
\ref{lemmaI} and Lemma \ref{lemmaII}.

Let $\pi : Z_{2^\nu} \times Z_\ell \rightarrow Z_{{2^\nu}\ell}$ and
$\pi' : Z_{2^{\nu+1}} \times Z_\ell \rightarrow Z_{{2^{\nu+1}}\ell}$
be the isomorphisms deduced from the Chinese reminder theorem.
\begin{lemma}\label{lemmaI}
  Suppose that $(c_v)_{v \in Z_{{2^\nu}\ell}}$ is a closed point of
  $V_{J'}$.  For any $i \in Z_\ell$ let $P_i$ be the closed point of
  $\proj_{\F_q}^{{2^\nu}^g-1}$ with homogeneous coordinates
  $(c_{\pi(k,i)})_{k \in Z_{2^\nu}}$. For all $i \in Z_\ell$, $P_i$ is a
  closed point of $V_{I_{\Theta_{2^\nu}}}$.
\end{lemma}

\begin{proof}
  It is enough to verify that for all $i \in Z_\ell$,
  $(c_{\pi(k,i)})_{k \in Z_{2^\nu}}$ satisfy the equations provided by the
  elements of $I_{\Theta_{2^\nu}}$.  For $i=0$ this is an immediate
  consequence of the hypothesis that $(a_k)_{k\in Z_{2^\nu}}$ is the theta
  null point associated to $\Theta_{2^\nu}$ and that by definition of $J'$,
  $a_k = c_{\pi(k,0)}$ for all $k \in Z_{2^\nu}$.

  Let $x,y,u,v \in Z_{2^{\nu+1}}$ which are congruent modulo
  $Z_{2^\nu}$. For any $i \in Z_\ell -\{ 0 \}$, we remark that $\pi(x,i),
  \pi(y,0), \pi(u,0), \pi(v,0) \in Z_{2^{\nu+1} \ell}$ are congruent modulo
  $Z_{{2^\nu} \ell}$. By definition of $I$ and the relations of
  Theorem \ref{riemannrelations}, $(c_{\pi(k,i)})_{k \in Z_{2^\nu}}$
  satisfy the relation
\begin{eqnarray*}
\big(\sum_{t \in Z_2} l(t) c_{\pi(x+y,i)+t} c_{\pi(x-y,i)+t}\big).\big(\sum_{t \in Z_2} l(t)
c_{\pi(u+v,0)+t} c_{\pi(u-v,0)+t}\big)= \\
=\big(\sum_{t \in Z_2} l(t) c_{\pi(x+u,i)+t} c_{\pi(x-u,i)+t}\big).\big(\sum_{t \in Z_2} l(t)
c_{\pi(y+v,0)+t} c_{\pi(y-v,0)+t}\big), 
\end{eqnarray*}
for all $l \in \hat{Z}_2$.

Taking care of the fact that $c_{\pi(k,0)}=a_k$ for all $k \in
Z_{2^\nu}$, we deduce that the point with homogeneous coordinates
$(c_{\pi(k,i)})_{k \in Z_{2^\nu}}$ satisfy all the relations of Theorem
\ref{riemannquad} and as a consequence is a closed point of
$V_{I_{\Theta_{2^\nu}}}$.
\end{proof}
Let $(c_v)_{v \in Z_{{2^\nu}\ell}}$ be a closed point of $V_{J'}$. Applying
Lemma \ref{lemmaI}, we denote by $P_i$ the closed point of
$V_{I_{\Theta_{2^\nu}}}$ with homogeneous coordinates $(c_{\pi(k,i)})_{k\in
  Z_{2^\nu}}$.

\begin{lemma}\label{lemmaII}
  The closed point $P_1$ is a $\ell$-torsion point of $A$.
  Moreover the application $\phi$ from the set of geometric points of
  $V_{J'}$ to $(\overline{\F}_q)^\ell$ defined by $\phi :
  \overline{\F}_q^{{2^\nu}\ell} \rightarrow \overline{\F}_q^{\ell}$,
  $(c_{j})_{j\in Z_{{2^\nu}\ell}} \mapsto (c_{\pi(k,1)})_{k \in Z_{2^\nu}}$ is
  injective.
\end{lemma}
\begin{proof}
  We are going to prove inductively on $i\in {1, \ldots \ell}$ that on
  the abelian variety $V_{I_{\Theta_{2^\nu}}}$ the point $i.P_1$ is
  equal to the point $P_i$. Applying this result for $i=\ell$, we
  obtain that $\ell P_1 = P_\ell = P_0$ and $P_0$ is the $0$ point of
  $A$ which means that $P_1$ is a $\ell$-torsion point of $A$.

  The induction hypothesis is clear for $i=1$. Let $(\theta_u)_{u\in
    Z_{2^\nu}}$ be the basis of global section of $\pol^{2^\nu}$ defined by
  $\Theta_{2^\nu}$. We suppose that for all $1 \leq j \leq i-1$, there
  exists an isomorphism $\xi_{j.P_1}: \pol^{2^\nu}_{j.P_1} \simeq
  \mathscr{O}_{A,j.P_1}$ such that $(\theta_k(j.P_1,
  \xi_{j.P_1}))_{k \in Z_{2^\nu}}=(c_{\pi(k,j)})_{k \in Z_{2^\nu}}$.  We have to
  prove that there exists an isomorphism $\xi_{i.P_1}: \pol^{2^\nu}_{i.P_1} \simeq
  \mathscr{O}_{A,i.P_1}$ such that $(\theta_k(i.P_1,
  \xi_{i.P_1}))_{k \in Z_{2^\nu}}=(c_{\pi(k,i)})_{k \in Z_{2^\nu}}$.

  Let $H_{2^\nu}=Z_{{2^\nu}} \times \hat{Z}_{2^{\nu-1}}$.  For all
  $x=(x',x'') \in H_{2^\nu}$, we let $$\vartheta_x = \sum_{t \in Z_{2^{\nu-1}}}
  x''(t) \theta_{x'+t}$$. Let $(x,y,u,v) \in H_{2^\nu}$ and
  $\tau=(\tau', \tau'') \in H_{2^\nu}$ such that $2\tau =x-y-u-v$. By
  the induction hypothesis, for $X\in \{(i-1)P, P, (i-2)P, 0 \}$, we
  have already a well defined isomorphisms $\pol_X^{{2^\nu}} \simeq
  \mathscr{O}_{A,X}$. We choose any isomorphism $\xi_{i.P_1}:
  \pol^{2^\nu}_{i.P_1} \simeq \mathscr{O}_{A,i.P_1}$.

  By applying Theorem \ref{classicalbis}, we deduce a relation
\begin{equation}\label{eq1}
\begin{split}
  \vartheta_{x}(i.P,\xi_{i.P})\vartheta_{y}((i-2).P,\xi_{(i-2).P}) \vartheta_{u}(0,\xi_0)
  \vartheta_{v}(0,\xi_0) =\\
=  \lambda \frac{1}{2^g} \sum_{t \in H_2} A(2t')
  \vartheta_{x-\tau+t}((i-1).P,
  \xi_{(i-1).P}) \vartheta_{y+\tau+t}((i-1).P,\xi_{(i-1).P}) \\
  \vartheta_{u+\tau+t}(P, \xi_P)
  \vartheta_{v+\tau+t}(P, \xi_P),
\end{split}
\end{equation}
where $t=(t',t'')$, $A=\tau''+t''$ and 
$\lambda \in \overline{\F}^*_q$ does not depend on the choice of
$(x,y,u,v) \in H_{2^\nu}$.

On the other side, denote by $H_{{2^\nu} \ell}=Z_{{2^\nu} \ell} \times
\hat{Z}_{2^{\nu-1}}$.  In the following we identify $H_{{2^\nu} \ell}$
with the Cartesian product $H_{2^\nu} \times Z_{\ell}$. For all
$x=(x',x'') \in H_{{2^\nu}\ell}$, we let $d_x = \sum_{t \in Z_2}
x''(t) c_{x'+t}$. Set $x_1=(x,i), y_1=(y,i-2), u_1=(u,0), v_1=(v,0)$.
Let $\tau \in H_{{2^\nu}}$ be such that $2\tau = x -y - u - v$ and
$\tau_1 = \tau \times \{ 1 \} \in H_{{2^\nu} \ell}$. We remark that $2
\tau_1 = x_1 -y_1 -u_1-v_1$ and by applying Theorem \ref{classical},
we get a relation deduced from the definition of $I$
\begin{eqnarray}\label{eq2}
  d_{x_1}d_{y_1} d_{u_1}d_{v_1} =
  \frac{1}{2^g} \sum_{t \in H_2} (\tau'' + t'')(2t')
  d_{x_1 - \tau_1 + t} d_{y_1 + \tau_1 +t} d_{u_1+ \tau_1+t} d_{v_1+
    \tau_1 +t}.
\end{eqnarray}
where $t=(t',t'') \in H_2$.

By the recurrence hypothesis and by the construction of the quadruples
$(x_1, y_1, u_1, v_1)$, we have for all $t\in H_2$, $d_{x_1 -\tau_1+t}
= \vartheta_{x-\tau+t}((i-1).P,\xi_{(i-1).P})$, $d_{y_1+\tau+t} =
\vartheta_{y+\tau+t}((i-1).P,\xi_{(i-1).P})$, $d_{u_1+\tau_1+t} =
\vartheta_{u+\tau+t}(P,\xi_P)$, $d_{v_1+\tau_1+t}
=\vartheta_{v+\tau+t}(P,\xi_P)$. In the same way, on the left hand
side of (\ref{eq2}), we have $d_{y_1} =
\vartheta_{y}((i-2).P,\xi_{(i-2).P})$, $d_{u_1} =
\vartheta_{u}(0,\xi_0)$ and $d_{v_1}=\vartheta_{v}(0,\xi_0)$.

There exists $u_0 \in H_{2^{\nu}}$ such that $\vartheta_{u_0}(0,
\xi_0)\neq 0$. We can take $u=v=u_0$ in Equations (\ref{eq1}) and
(\ref{eq2}) and we deduce immediately that
$d_{x_1}=\vartheta_{x}(i.P,\xi_{i.P})$. By taking all possible values
of $x_1''$ in $x_1=(x_1',x_1'')$, we obtain that
\begin{itemize}
\item for all $k \in Z_{2^\nu}$, $c_{\pi(k,i)}$ is uniquely determined from the knowledge of
  $c_{\pi(k,j)}$ for $j \leq i-1$;
\item 
for all $k \in Z_{2^\nu}$,
$c_{\pi(k,i)}=\theta_k(i.P,\xi_{i.P})$ modulo multiplication by a constant factor
independent of $k$ that we normalise to $1$ by choosing a certain $\xi_{i.P}$.
\end{itemize}
\end{proof}

\begin{proof}
  By the preceding lemma, $(c_{\pi(k,1)})$, being the homogeneous
  coordinate of a $\ell$-torsion point of $A$, can only assume
  a finite number of value up to a multiplication by a constant factor
  and the data of $(c_{\pi(k,1)})$ defines a unique solution of the
  system associated to $J'$. In order to finish the proof, we only
  have to show that $J'$ is not a homogeneous ideal but this is
  something clear from the definition.
\end{proof}

\section{Practical implementation and examples}
\label{implementation}

The proved version of the algorithm involve the resolution of
algebraic systems which makes it not suitable for practical
applications.  We have implemented the heuristic version of the
algorithm for the case of genus $1$ and genus $2$~\cite{algo}. For the
genus $2$ implementation, using a special purpose Groebner basis
algorithm it is possible to solve easily the algebraic system of the
initialisation phase.

\paragraph{A genus 1 characteristic 5 example.}
Let $\F_{5^8}$ be represented by the quotient $\F_5[X]/(P)$ where
$P(X)= X^8 + X^4 + 3X^2 + 4X + 2$ and let $u$ be the image of $X$ in
$\F_{5^8}$ via the above isomorphism.
Let $E$ be the ordinary elliptic curve given by the Weierstrass
equation $$y^2 = x^3 + x^2 + 3x.$$
After the initialisation phase we obtain the following six theta
constants
{\tiny
$$[ 1, 4, u^{32552}, u^{309244}, u^{211588}, u^{32552} ].$$}
We consider $\Z_{5^8}$ given as the unramified extension of the
$5$-adic integers $\Z_5$ defined by the integer polynomial $X^8 + X^4 + 3X^2 + 4X + 2$ and denote
by $z$ the image of $X$ in $\Z_{5^8}$. 
After the lift phase we get the following lifted theta constants to
precision $5$
{\tiny
\begin{eqnarray*}
[1, -1460z^7 - 10z^6 - 785z^5 + 715z^4 - 555z^3 + 420z^2
    - 1035z - 1116, \\
-1449z^7 - 819z^6 + 396z^5 + 746z^4 + 1108z^3 +
    648z^2 + 546z - 1189, \\
1438z^7 - 1497z^6 + 1548z^5 - 777z^4 + 354z^3 -
    876z^2 + 998z + 1029, \\
1449z^7 + 819z^6 - 396z^5 - 746z^4 - 1108z^3 -
    648z^2 - 546z - 868, \\
-1504z^7 + 101z^6 + 741z^5 + 591z^4 - 957z^3 -
    492z^2 - 1109z - 834] \\
\end{eqnarray*}}
where $z$ is a generator 

After the norm phase we obtain $1054$ as the number of rational points
on $E$.

\paragraph{A genus $2$ characteristic $3$ example}
Let $\F_{3^{28}}$ be represented by the quotient $\F_3[X]/(P)$ where
{\tiny $$P(X)= X^{28} + 2X^{14} + X^{13} + X^{12} + 2X^{11} + X^{10} + X^9 + X^8 +
    2X^6 + 2X^4 + X^3 + 2$$}
and let $w$ be the image of $X$ in
$\F_{5^8}$ via this isomorphism.  
Let $H$ be the ordinary hyperelliptic curve given by the affine equation
{\tiny
\begin{eqnarray*}
y^2 &=  & x^6 + (w^{18} + w^{17} + w^{16} + w^{11} + w^{10} +
    w^9 + w^8 + w^7 + 2w^5 + 2w^2 + w)x^5 \\
&& + (w^{19} + 2w^{17} + 2w^{16} + w^{13} +
    w^{11} + w^{10} + 2w^8 + 2w^7 + w^6 + 2w^4 + w + 2)x^4 \\
&& + (2w^{19} + 2w^{18} +
    2w^{17} + 2w^{15} + 2w^{14} + 2w^{12} + 2w^{11} + 2w^{10} + 2w^9 +
    w^7 + 2w^6 \\
&& \hspace{0.5cm} + w^5 + 2w^4 + w^3 + w + 1)x^3 \\
& & + (w^{19} + 2w^{18} + 2w^{16} + 2w^{13} + w^{12} +
    w^{10} + 2w^9 + w^8 + w^6 + 2w^2 + 1)x^2  \\
& &+ (w^{19} + 2w^{18} + w^{17} + 2w^{15} +
    2w^{14} + w^{13} + w^{12} + w^{11} + 2w^9 + w^8 + w^6 \\
&& \hspace{0.5cm} + 2w^5 + 2w^4 + w^3 +
    2w^2 + 2)x \\
& & + w^{19} + 2w^{16} + w^{15} + w^{14} + w^{12} + 2w^8 + w^7 + w^6 + w^4
    + 2w^3 + w^2 + w + 1
\end{eqnarray*}}
First, we compute the following level $2$ theta constants
{\tiny
\begin{eqnarray*}
  x_{00} &= &  w^{19} + w^{18} + 2*w^{15} + w^{14} + w^{12} + 2w^{10} + w^7 + 2w^6 + 2w^5 +
        2w^4 + w^3 + w + 2 \\
  x_{03} & = &  w^{19} + 2w^{17} + 2w^{16} + 2w^{15} + 2w^{14} + 2w^{13} + w^{12} + 2w^{11} +
        2w^{10} + 2w^9 + 2w^8 + w^7 + \\ & & w^6 + 2w^5 + w^4 + w^3 + 2w^2
        + 2w + 2 \\
  x_{30} & = &  w^{19} + 2w^{18} + w^{17} + w^{16} + 2w^{15} + 2w^{14} + 2w^{13} + w^{12} +
        2w^{11} + 2w^{10} + 2w^9 \\ && + 2w^7 + 2w^3 + w^2 + 2 \\
  x_{33} & = &  2w^{19} + 2w^{18} + w^{17} + 2w^{15} + 2w^{13} + 2w^{12} + w^{10} + 2w^9 +
        w^8 + w^6 + 2w^4 + 2w^3 + w^2 + 2w + 1.
\end{eqnarray*}
}
After the Groebner basis step, we obtain the following list of theta
constants
{\tiny
\begin{eqnarray*}
  0 & = & x_{01}+w^{18}+w^{16}+w^{15}+2w^9+w^8+w^7+w^6+2w^5+w^4+2 \\
  0 & = &
  x_{02}+w^{19}+2w^{17}+2w^{16}+2w^{15}+w^{14}+w^{13}+w^{12}+w^{11}+w^{10}+w^9+\\ & & w^7+2w^5+w^4+w^3+w^2+w+2\\
  0 & = & x_{10} +
  2w^{19}+2w^{18}+w^{17}+2w^{14}+2w^{13}+2w^{12}+w^{11}+w^{10}+w^9+2w^8\\ && +2w^6+2w^4+w^3+2w^2+2 \\
  0 & = &x_{11} +
  2w^{19}+w^{16}+w^{15}+2w^{14}+2w^{12}+2w^{11}+2w^{10}+2w^9+w^7+w^5
  \\ && +w^4+w^3+2w^2+2w+2 \\
  0 & = &
  x_{12}+w^{19}+2w^{18}+2w^{17}+w^{16}+2w^{15}+w^{14}+w^{13}+w^{12}+2w^{10}+2w^9+w^8+\\ & & 2w^7+2w^6+w^4+2w^3+2w^2+2w+1 \\
  0 & = & x_{13} +w^{18}+w^{17}+2w^{14}+2w^{13}+w^9+2w^6+2w^5+1 \\
  0 & = &
  x_{20}+w^{19}+w^{18}+2w^{16}+w^{15}+w^{14}+w^{13}+w^{12}+w^{11}+2w^{10}+w^9+2w^7 \\ && +2w^6+w^4+w^3+w+2 \\ 
  0 & = & x_{21}+w^{19}+w^{17}+w^{16}+w^{15}+2w^{14}+2w^{12}+w^{10}+w^5+w^3+w^2+w+2\\
  0 & = &
  x_{22}+2w^{19}+w^{17}+2w^{16}+2w^{15}+w^{13}+w^{12}+2w^{11}+2w^{10}
  \\ && +2w^9+w^8+2w^7+2w^5+w^4+w^2+w+1 \\
  0 & = &  x_{23}+w^{18}+2w^{14}+w^{12}+2w^{11}+2w^{10}+w^8+w^6+w^5+w^2+w+1 \\
  0 & = & x_{31}+w^{18}+w^{17}+w^{16}+2w^{15}+2w^{13}+2w^{11}+w^9+w^8+w^7+2w^4+2w^3+2w^2+2\\
  0 & = & x_{32}+2w^{19}+2w^{18}+2w^{17}+2w^{16}+w^{15}+2w^{14}+w^{13}+w^{12}+w^{11}+w^9+w^7+w^6+2w^2+w\\
  0 & = & x_{41} +w^{18}+2w^{16}+2w^{15}+2w^{13}+w^{12}+w^{11}+w^{10}+2w^9+w^8+w^7+2w^6+w^4+2w\\
  0 & = & x_{42}+2w^{16}+2w^{14}+2w^{12}+2w^{10}+w^9+2w^8+2w^6+2w^5+2w^4+w^3+w^2+w+2\\
  0 & = &
  x_{51}+2w^{18}+w^{17}+2w^{16}+2w^{15}+2w^{13}+w^{11}+w^{10}+w^9+2w^8+w^7+2w^6+\\ & & 2w^5+2w^4+2w^3+w^2+2w\\
  0 & = & x_{52}+2w^{17}+2w^{16}+2w^{15}+w^{14}+2w^{12}+w^{10}+2w^9+2w^7+w^6+w^5+2w^4+w^3+2w^2+w\\
\end{eqnarray*}}
After the norm phase, we obtain as a product of the Eigenvalues of the
Frobenius morphism which are units modulo $3$ the number
{\tiny
$$202395421016914130938488532$$}
to precision $56$.
From here, we can recover the polynomial $\chi_F$ which is
{\scriptsize
\begin{eqnarray*}
\chi_F(X)= X^4 + 19612X^3 - 4108934426X^2 + 68382815672412X +
    12157665459056928801.
\end{eqnarray*}}

\section*{Conclusion}
We have given an algorithm with quasi-quadratic time and quadratic
space complexity with respect to the size of the base field to compute
the number of points of a hyperelliptic curve whose Jacobian is
ordinary and absolutely simple.

In fact, we have given two versions of our algorithm, one with proved
complexity bound and a bad practical behaviour and a heuristic one
which behaves very well in practice.

\newcommand{\etalchar}[1]{$^{#1}$}

\end{document}